\documentclass{article}

\usepackage{arxiv}

\usepackage[utf8]{inputenc} 
\usepackage[T1]{fontenc}    
\usepackage{hyperref}       
\usepackage{url}            
\usepackage{booktabs}       
\usepackage{amsfonts}       
\usepackage{nicefrac}       
\usepackage{microtype}      
\usepackage{doi}
\usepackage{style}
\usepackage{enumitem}

\usepackage[mode=buildnew]{standalone}
 
\title{Hydro-mechanical Model for Slope Stability Assessment: A polygonal stabilization-free discretization}

\author{ \href{https://orcid.org/0000-0001-8642-4258}{\includegraphics[scale=0.06]{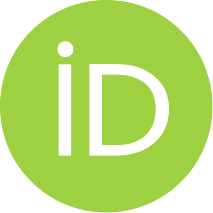}\hspace{1mm}Stefano~Berrone} \\
	Dipartimento di Scienze Matematiche\\
	``G. L. Lagrange''\\
	Politecnico di Torino, Italy \\
	\texttt{stefano.berrone@polito.it} \\
    \And
	\href{https://orcid.org/0000-0001-9548-5300}{\includegraphics[scale=0.06]{orcid.pdf}\hspace{1mm}Francesca~Marcon} \\
	Dipartimento di Scienze Matematiche\\
	``G. L. Lagrange''\\
	Politecnico di Torino, Italy \\
	\texttt{francesca.marcon@polito.it} \\
	\And
	\href{https://orcid.org/0000-0002-8540-3639}{\includegraphics[scale=0.06]{orcid.pdf}\hspace{1mm}Gioana~Teora} \\
	Dipartimento di Scienze Matematiche\\
	``G. L. Lagrange''\\
	Politecnico di Torino, Italy \\
	\texttt{gioana.teora@polito.it} \\
}

\renewcommand{\shorttitle}{Hydro-mechanical Model for Slope Stability Assessment}

\hypersetup{
pdftitle={Hydro-mechanical Model for Slope Stability Assessment: A polygonal stabilization-free discretization},
pdfsubject={q-bio.NC, q-bio.QM},
pdfauthor={Stefano~Berrone, Francesca~Marcon, Gioana~Teora},
pdfkeywords={},
}

\begin{document}
\maketitle

\begin{abstract}
Rainfall-induced landslides are governed by the interaction between subsurface water flow and soil mechanics, requiring robust numerical methods for the simulation of variably saturated porous media. In this work, we consider a semi-coupled hydro-mechanical model based on Richards' equation and linear elasticity and propose a numerical framework based on a stabilization-free Virtual Element Method for its spatial discretization. The proposed approach naturally accommodates general polygonal meshes while avoiding problem-dependent stabilization terms, whose design may become challenging when heterogeneous and strongly non-linear coefficients are involved. The approach is combined with a mass-lumping technique to improve stability in the treatment of the storage term and with Nitsche's method to weakly impose seepage-face and infiltration boundary conditions, allowing for the automatic switching between Neumann and Dirichlet conditions. Time integration is performed using the backward Euler scheme, while non-linearities are handled through a Picard iteration. Numerical experiments demonstrate the stability and robustness of the proposed methodology and show its effectiveness in simulating rainfall infiltration and evaluating slope stability through the Local Factor of Safety.
\end{abstract}



\section{Introduction}

Landslides are geological processes involving the downslope movement of soil and rock masses, often triggered by intense or prolonged rainfall events that increase soil saturation and reduce slope stability \cite{Abbasov2024}. Rainwater infiltration modifies both the unit weight of the soil and the pore water pressure, thereby altering the stress distribution within hill-slopes and reducing the available shear strength \cite{MoradiHuisman2024}. In particular, an increase in pore water pressure decreases matric suction and effective stress, leading to a reduction in soil cohesion and, consequently, to a higher susceptibility to slope failure \cite{Moradi2018}.

To investigate these processes, hydro-mechanical multi-physics models are commonly employed, coupling subsurface flow and soil mechanics \cite{Lu2012,Moradi2018}. Since soils are generally unsaturated or partially saturated, water flow is modelled using Richards' equation, which combines Darcy's law, mass conservation, and constitutive relationships linking saturation and permeability to the pressure-head \cite{Abbasov2024,Moradi2018}. The mechanical behaviour is described by the linear momentum equilibrium equation together with a linear elasticity constitutive relation. The resulting mathematical model describes the quasi-static consolidation of variably saturated porous media, commonly referred to as unsaturated poro-elasticity \cite{Both2021}.

Slope stability is commonly assessed through the Local Factor of Safety ($\rm{LFS}$), a Coulomb stress-field-based indicator defined as the ratio between the Coulomb stress for the potential failure state and the Coulomb stress for the current state of stress under the Mohr-Coulomb criterion \cite{Lu2012}. The $\rm{LFS}$ is evaluated pointwise throughout the computational domain and depends on the effective stress tensor, soil cohesion, and internal friction angle. Unlike traditional limit-equilibrium approaches, the $\rm{LFS}$ does not require prior assumptions regarding the geometry or location of the failure surface. Consequently, it can be naturally computed on unstructured meshes, providing detailed information on the onset of instability and on the spatial distribution of potential failure zones \cite{Moradi2018,Abbasov2024,Lu2012}.

Among the numerical techniques available for slope stability analysis, the Finite Element Method (FEM) is one of the most widely adopted because of its efficiency, versatility, and relatively low computational cost \cite{LinZheng2020}. Although FEM-based approaches have been extensively developed and refined, their applicability to highly complex geometries remains challenging, and the accuracy of the numerical solution is often sensitive to mesh quality. Moreover, generating high-quality meshes may become particularly difficult for domains with intricate geometrical features, such as stony soil slopes, where rock blocks of different sizes and shapes are randomly distributed throughout the soil matrix \cite{Teora2024,GrappeinTeora2025}.

An attractive alternative is provided by the Virtual Element Method (VEM), which naturally accommodates general polygonal meshes, including non-convex elements, hanging nodes, collapsing nodes, and other more complex configurations \cite{SunLin2020}. This flexibility greatly simplifies the discretization of complex geometries while preserving good approximation properties. Owing to these advantages, VEM has been successfully applied to a wide range of problems, including linear elasticity \cite{DaVeiga2015}, fracture mechanics \cite{Benedetto2018,Berrone2013}, contact problems \cite{Wriggers2016}, poro-mechanics \cite{BurgerBaier2020}, and stony soil modelling \cite{SunLin2020,GrappeinTeora2025}. Nevertheless, standard VEM formulations require the introduction of problem-dependent stabilization terms to recover coercivity, making their application to strongly non-linear problems, such as Richards' equation, considerably more challenging \cite{GrappeinTeora2025}.

To overcome this limitation, we employ a stabilization-free Virtual Element Method (SFVEM) \cite{Marcon2024} for the spatial discretization of the non-linear semi-coupled hydro-mechanical model. The method is combined with a mass-lumping strategy \cite{Paulino2020}, which eliminates the need for stabilization terms also in the storage contribution while mitigating spurious oscillations at the infiltration front. Furthermore, seepage-face and infiltration boundary conditions (BCs) are imposed through Nitsche's method \cite{Mika2009,TagliabueQuarteroni2016}, allowing for the automatic transition between Neumann and Dirichlet boundary conditions according to the local hydraulic state \cite{Dolejsi2019}. Time discretization is performed using the backward Euler scheme, whereas the non-linearities arising from Richards' equation are handled through a Picard iterative procedure.

The outline of the paper is as follows. Section \ref{sec:hydromechanical_problem} introduces the hydro-mechanical model, defines the Local Factor of Safety, and presents the seepage-face and infiltration boundary conditions. Section \ref{sec:numerical_discretization} describes the proposed numerical framework, which combines the stabilization-free Virtual Element Method, the backward Euler time discretization, and the Picard iterative scheme for the solution of the hydro-mechanical problem. The same section also discusses the post-processing procedure adopted to compute the Local Factor of Safety. Section \ref{sec:nitshce_method} presents the Nitsche's formulation for the treatment of seepage-face and infiltration boundary conditions and provides the theoretical analysis establishing the stability of the resulting spatial discretization. Finally, Section \ref{sec:numerical_experiments} reports several numerical experiments designed to assess each component of the proposed methodology and to demonstrate its reliability and robustness in the simulation of semi-coupled hydro-mechanical problems.


\section{The hydro-mechanical model}\label{sec:hydromechanical_problem}

The following section presents the required mass and momentum balance equations that govern variably saturated flow in hill-slopes and couple hydraulic and mechanical processes.

We consider a poro-elastic medium occupying the open, connected, and bounded domain $\Omega \subset \R^2$ with Lipschitz boundary $\Gamma \coloneq \partial \Omega$. Let $\finaltime > 0$ denote the final time and $(0,\finaltime)$ denote the time interval of interest. Under the assumption of infinitesimal deformations of the skeleton, the poro-elastic medium can be approximated as fixed in time \cite{Both2021}. Let us denote by $Q_{\finaltime} := \Omega \times (0,\finaltime)$ the space-time domain. Let $\xx \in \Omega$, then we write $\xx = (x, z)$.

In the following, the notation ($[\cdot]$) denotes the physical dimension of the quantity that precedes it. We use $L$, $M$, and $T$ to represent the dimensions of length, mass, and time, respectively, while $[-]$ denotes a dimensionless quantity.

We consider the Richards' equation \cite{Richards1931} to model water flow, 
i.e.
\begin{equation}
\begin{cases}
    \frac{\partial \wc(\ph)}{\partial t} + \nabla \cdot \left(-K(\ph) (\nabla\ph + \ee_z)\right) = f & \text{in } Q_{\finaltime},\\
    \ph = 0 & \text{on } \Gamma_D^{\ph}\times (0, \finaltime),\\
    K(\ph)(\nabla \ph + \ee_z) \cdot \nn = 0 & \text{on } \Gamma_N^{\ph}\times (0, \finaltime),\\
    K(\ph)\left(\nabla \ph + \ee_z\right) \cdot \nn + \gamma(\ph) \ph = G_N + \gamma(\ph) g_D & \text{on } \Gamma_R^{\ph} \times (0, \finaltime),\\
    \ph(\cdot, 0) = \ph_0 & \text{in } \Omega,
\end{cases}
\label{eq:richards}
\end{equation}
where $\ph\ [L]$ is the water pressure-head, $\wc\ [-]$ is the volumetric water content, $f\in \leb{2}{Q_{\finaltime}}\ [T^{-1}]$ represents the sink/source term, the symbol $z\ [L]$ denotes the vertical coordinate with the $z$-axis direction $\ee_z$ oriented against the gravity direction, $K\ [LT^{-1}]$ represents the hydraulic conductivity, and $\ph_0 \in \leb{2}{\Omega} \ [L]$ is the initial condition. Moreover, we consider homogeneous Dirichlet and Neumann boundary conditions on $\Gamma^{\ph}_D$ and $\Gamma^{\ph}_{N}$, respectively,
and a time-varying boundary parts on $\Gamma_R^{\ph}$, such that these parts are mutually
disjoint and $\Gamma_N^{\ph} \cup \Gamma_D^{\ph} \cup \Gamma_R^{\ph} = \Gamma$. On the non-homogeneous time-varying boundary, we prescribe suitable boundary conditions represented by functions $g_D \in \leb{2}{0,\finaltime; \sob{1/2}{\Gamma^{\ph}_R}} \ [L]$, and $G_N \in \leb{2}{0,\finaltime; \leb{2}{\Gamma^{\ph}_R}}  \ [LT^{-1}]$. In particular, this time-varying boundary condition is used here to prescribe seepage and infiltration boundary conditions (see Section \ref{sec:seepage} for further details). 
On $\Gamma^{\ph}_D$ and $\Gamma^{\ph}_{N}$, we assume homogeneous boundary conditions for the pressure-head, for simplicity. However, more complex BCs can also be incorporated into the model with standard techniques.
Finally, $\gamma = \gamma(\ph)\ [T^{-1}]$ acts as a permeability coefficients for the surface $\Gamma^{\ph}_R$ that depends on the soil saturation. We remark that the term ``time-varying'' indicates that the type of boundary condition changes in time, switching between Neumann and Dirichlet conditions according to the value of $\gamma$, whereas the boundary geometry is time-independent.

Moreover, we consider the following assumptions to hold true.

\begin{assum}[\cite{List2016, Dolejsi2019}]
Let us assume the following assumptions on function coefficients:
    \begin{enumerate}[label={(A\arabic*)}]
        \item The water content $\wc = \wc(\ph)$ is Lipschitz continuous and monotonically non-decreasing with derivative $C(\ph) \coloneq \frac{\partial \wc(\ph)}{\partial \ph} \geq 0$, commonly referred to as the capacity term. If $\ph > 0$, then $C(\ph) = 0$, and the Richards' equation degenerates (\textit{fast-diffusion} type of degeneracy). This is a common situation when considering the infiltration process, as the ones considered in Section \ref{sec:numerical_experiments}.
        \item The hydraulic conductivity $K=K(\ph)$ is a positive, non-decreasing, Lipschitz continuous, and there exist two constants $K_m$ and $K_M$ such that
        \begin{equation*}
            0 < K_m \leq K(z) \leq K_M < \infty\quad \forall z \in \R.
        \end{equation*}
        Typically, for $\ph \rightarrow -\infty$, $K(\ph) \rightarrow 0$ and the Richards' equation degenerates (\textit{slow-diffusion} type of degeneracy).
    \end{enumerate}
\end{assum}

Let us introduce the effective degree of saturation $\effsat\ [-]$
\begin{equation}
    \effsat(\ph) = \frac{\wc(\ph) - \wc_r}{\wc_s - \wc_r},
    \label{eq:effsat}
\end{equation}
where $\wc_r\ [-]$ and $\wc_s\ [-]$ are the residual and the saturated volumetric water content, respectively.

In the present manuscript, two different experimental laws are used to describe the soil water retention curves, i.e. $K=K(\ph)$ and $\wc=\wc(\ph)$. The first law is based on Mualem and van Genuchten \cite{VanGenuchten1980}, which defines 
\begin{equation}
    \begin{gathered}
\effsat(\ph) = \begin{cases}
\frac{1}{[1 + (a \vert \ph \vert)^n]^m} & \text{if } \ph < 0,\\
1 & \text{if }  \ph \geq 0,
\end{cases}\\
K(\ph) = \begin{cases}
K_s  \effsat(\ph)^{l}\left[1 - \left(1- \effsat(\ph)^{\frac{1}{m}}\right)^m\right]^2 & \text{if } \ph < 0,\\
K_s & \text{if }  \ph \geq 0,
\end{cases}
\end{gathered}
\label{eq:van:parameters}
\end{equation}
where $a\ [L^{-1}]$ represents the inverse of the air entry suction, $n\ [-]$ and $m = 1 - \frac{1}{n}$ depend on the pore size distribution, $K_s\ [LT^{-1}]$ is the saturated hydraulic conductivity and $l\ [-]$ is a tortuosity parameter, usually set equal to $2$. 
The second approach is based on Brooks and Corey \cite{Brooks1964}. It defines
\begin{equation}
    \begin{gathered}
    \effsat(\ph) = \begin{cases}
        \vert \alpha_{BC} h \vert^{-n_{BC}} & \text{if } \alpha_{BC} \ph < -1,\\
        1 & \text{if } \alpha_{BC} \ph \geq -1,
    \end{cases}\\
    K(\ph) = K_s \effsat(\ph)^{l_{BC} + 2 + \frac{2}{n_{BC}}},
\end{gathered}
\label{eq:brooks:parameters}
\end{equation}
where  $\alpha_{BC}\ [L^{-1}]$ is the air entry pressure-head, $n_{BC}\ [-]$ is the soil pore distribution index, and $l_{BC}\ [-]$ is a parameter that depends on tortuosity.

The linear momentum equilibrium equation is expressed as follows \cite{Lu2012,Moradi2018}:
\begin{equation}
    \begin{cases}
        \nabla \cdot \stress(\uu) + \bb(\ph) = \bm{0} & \text{in } Q_{\finaltime},\\
        \stress(\uu) = 2 \mu \strain(\uu) + \lambda \div \uu \II & \text{in } Q_{\finaltime},\\
        \strain(\uu) = \frac{1}{2} (\nabla \uu + (\nabla \uu)^T)& \text{in } Q_{\finaltime},\\
        \uu = \bm{0} & \text{on } \Gamma, \quad \forall t \in (0, \finaltime),
    \end{cases}
    \label{eq:momentum_balance}
\end{equation}
where $\uu\ [L]$ is the displacement, $\strain(\uu)\ [-]$ is the total symmetric strain  tensor, $\stress\ [ML^{-1}T^{-2}]$ is the total stress tensor, $\mu\in\R^+ \ [ML^{-1}T^{-2}]$ and $\lambda\in\R^{+}\ [ML^{-1}T^{-2}]$ are the Lamé coefficients, $\II\ [-]$ represents the second order identity tensor, and $\bb \in\vleb{2}{\Omega}{2}\ [M L^{-2}T^{-2}]$ is a body force vector. We note that the dependence $\bb=\bb(\ph)$ introduces the coupling between the hydraulic and mechanical problems, accounting for the additional loads induced in the soil by variations in water content.

We assume homogeneous boundary conditions for the displacement, for simplicity. However, more complex BCs can also be incorporated into the model with standard techniques.


\subsection{The variational formulation}

Let us introduce the following variational spaces
\begin{gather*}
    \VP \coloneq \{v \in \sob{1}{\Omega}:\ v = 0 \text{ on }{\Gamma_D^{\ph}}\}, \quad  \VM \coloneq \vsob[0]{1}{\Omega}{2}.
\end{gather*}
The variational formulation of the coupled hydro-mechanical model reads as: \textit{Find $\ph \in \leb{2}{0,\finaltime; \VP}$ with $\frac{\partial \ph}{\partial t} \in \leb{2}{0,\finaltime; \VP'}$ and $\uu \in \leb{2}{0, \finaltime; \VM}$ such that}
\begin{equation}
    \begin{cases}
    \begin{aligned}
        \sdual{C(\ph(t))\frac{\partial \ph(t)}{\partial t}}{v} &+ \scal[\Omega]{K(\ph(t)) (\nabla \ph(t) + \ee_z)}{\nabla v} \\
        &\quad - \scal[\Gamma_R^{\ph}]{K(\ph(t)) (\nabla \ph(t) + \ee_z) \cdot \nn}{v} = \scal[\Omega]{f(t)}{v} 
    \end{aligned}
         & \quad \forall v \in \VP \text{ and for a.e. } t \in (0, \finaltime)\\
        \scal[\Omega]{2 \mu \strain(\uu(t))}{\strain(\vv)} + \scal[\Omega]{\lambda \div \uu(t)}{\div \vv} + \scal[\Omega]{\bb(\ph(t))}{\vv} = 0 & \quad \forall \vv \in \VM \text{ and for a.e. } t \in (0, \finaltime)\\
        \ph(0) = \ph_0
    \end{cases}
    \label{eq:var-problem}
\end{equation}
where $\sdual{}{}$ denotes the duality pairing between $\VP$ and its dual space $\VP'$, whereas $\scal[\Omega]{}{}$ denotes the inner product in $\leb{2}{\Omega}$. More precisely, given two scalar functions $p,q \in \leb{2}{\Omega}$, two vector fields $\bm{a},\ \vv \in \vleb{2}{\Omega}{2}$ and two tensor fields $\bm{T},\ \bm{\sigma} \in \vleb{2}{\Omega}{2 \times 2}$, we denote by
\begin{align*}
\scal[\Omega]{p}{q} = \int_{\Omega} p q,\quad\scal[\Omega]{\bm{a}}{\vv} = \int_{\Omega} \bm{a}\cdot \vv,\quad \scal[\Omega]{\bm{T}}{\bm{\sigma}} = \int_{\Omega} \bm{T} : \bm{\sigma}, 
\end{align*}
where $\bm{T} : \bm{\sigma} := \sum_{i,j = 1}^n \bm{T}_{ij} \bm{\sigma}_{ij}$.

\subsection{The local factor of safety}

\begin{figure}[ht!]
    \centering
    \includegraphics[width=.8\textwidth]{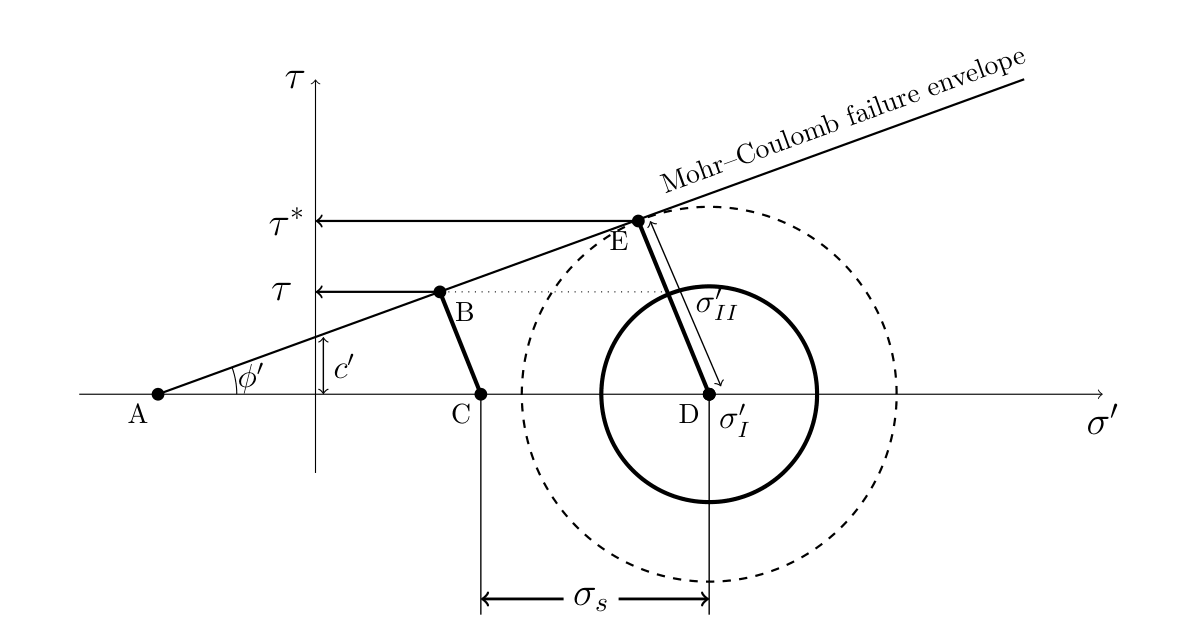}
    \caption{The Mohr circle (black solid circle) and the Mohr-Coulomb failure envelope used to evaluate the $LFS$ \cite{Lu2012}.}
    \label{fig:mohr-coulumb}
\end{figure}

Once the mass and momentum balance equations have been solved, the stability of the variably saturated hill-slope can be evaluated. Here, the Local Factor of Safety ($\rm{LFS}$) approach proposed by \cite{Lu2012} is used for stability evaluation, whose definition is based on the concept of effective stress. The effective stress in variably saturated soils, denoted by $\stress'\ [ML^{-1}T^{-2}]$, is given by Bishop \cite{Lu2010, Moradi2018}:
\begin{equation*}
    \stress' = \stress - p_a \II + \chi(\ph) (p_a - p_w)\II,
\end{equation*}
where $p_a\ [ML^{-1}T^{-2}]$ is the pore air pressure, $p_w\ [ML^{-1}T^{-2}]$ is the pore water pressure, and $\chi\ [-]$ is the Bishop parameter, which can be defined as
\begin{equation*}
    \chi(\ph)   = \begin{cases}
        1 & \ph \geq 0,\\
        \effsat(\ph) & \ph < 0,
    \end{cases} 
\end{equation*}
where $\effsat$ is defined in \eqref{eq:effsat}.
Assuming $p_a = 0$ provided by the surrounding atmospheric air and being $p_w = \rho_w g \ph$, we obtain
\begin{equation}
    \stress' =  \stress - \sigma_s(\ph) \bm{I},
    \label{eq:effective_stress}
\end{equation}
where $\sigma_s(\ph) = \chi(\ph) \rho_w g \ph \ [ML^{-1}T^{-2}]$ is the suction stress, which is always compressible for soils, $\rho_w\ [ML^{-3}]$ is the water density, and $g\ [LT^{-2}]$ is the gravity acceleration.

The computation of Local Factor of Safety helps to quantify, at each point within a slope, the proximity of the current stress state to failure and it is based on the effective stress tensor \eqref{eq:effective_stress}. 
For a linear elastic material, the shear strength and, thus, the Mohr-Coulomb failure envelope is defined as
\begin{equation*}
    \tau = c' + \sigma' \tan (\phi'),
\end{equation*}
where $\tau\ [ML^{-1}T^{-2}]$ is the shear stress, $\sigma'$ represents the normal component of the effective stress, $c'\ [ML^{-1} T^{-2}]$ is the effective soil cohesion, and $\phi'\ [\circ]$ is the effective friction angle. See Figure \ref{fig:mohr-coulumb} for a graphical representation.

More precisely, the LFS is defined as the ratio of the potential Coulomb stress $\tau^{\ast}\ [ML^{-1}T^{-2}]$ to the current state of shear stress $\tau$ in the failure direction, according to the Mohr-Coulomb criterion \cite{Lu2012}. For a given effective stress state, the LFS can be computed geometrically using the Mohr circle representation and the similarity between triangles $ABC$ and $ADE$ in Figure \ref{fig:mohr-coulumb}, yielding 
\begin{equation}
    \rm{LFS} = \frac{\tau^{\ast}}{\tau} = \frac{\cos(\phi') (c' +  \sigma'_{I} \tan(\phi'))}{\sigma_{II}'},
    \label{eq:LFS}
\end{equation}
where $(\sigma_{I}', 0)$ and $\sigma_{II}'$ are the centre and the radius of the Mohr circle, respectively. These quantities can be expressed as
\begin{gather*}
    \sigma_I' = \frac{\sigma_1'  + \sigma_3'}{2} = \frac{\sigma_1  + \sigma_3}{2}  - \sigma_s,\\
        \sigma_{II}' = \frac{\sigma_1'  - \sigma_3'}{2} = \frac{\sigma_1  - \sigma_3}{2}  ,
\end{gather*}
where $\sigma_1'$ and $\sigma_3'$ are the major and minor principal effective stress, while $\sigma_1$ and $\sigma_3$ are the major and minor principal total stress. This relationship follows from Equation \eqref{eq:effective_stress}.

Based on the Mohr-failure criterion,
\begin{itemize}
    \item $\rm{LFS} > 1$ indicates a stable soil,
    \item $\rm{LFS} = 1$ defines the stability threshold, 
    \item $\rm{LFS} < 1$ indicates that the soil may experience a failure.
\end{itemize}

We observe that if, for instance, the body force vector $\bb$ in \eqref{eq:momentum_balance} is independent of $\ph$ and it depends only on the soil self-weight, infiltration processes increase the suction stress $\sigma_s$, causing a leftward shift of the Mohr circle and consequently a reduction of the $\rm{LFS}$.

\subsection{Seepage and infiltration}\label{sec:seepage}

A seepage-face is the boundary between a saturated flow field and the atmosphere, or between a saturated flow field and a stream channel, where water is free to exit from the subsurface. This condition can be formulated in the following form \cite{Scudeler2017}:
\begin{equation*}
    \ph \leq 0,\quad  K(\ph)(\nabla \ph + \ee_z) \cdot \nn \leq 0,\quad \ph(K(\ph)(\nabla \ph + \ee_z) \cdot \nn ) = 0,\qquad \text{on } \Gamma_R^{\ph} \quad \forall t \in (0, \finaltime).
\end{equation*}
This means that if the if water is discharged from the domain (negative entering flux) the pressure-head $\ph$ must be atmospheric. Conversely, whenever the pressure-head is negative $\ph < 0$, no seepage outflow occurs. In other words, we either impose a homogeneous Dirichlet boundary condition ($\ph = g_D = 0$) or a homogeneous Neumann boundary condition ($K(\ph)(\nabla \ph + \ee_z) \cdot \nn = G_N = 0$).

Until water is not accumulated at the soil surface, i.e. no ponding occurs, infiltration (or snowmelt) can be prescribed as a Neumann boundary condition
\begin{equation*}
    K(\ph)(\nabla \ph + \ee_z) \cdot \nn = G_N\quad \text{on } \Gamma_{R}^{\ph} \quad \forall t \in (0, \finaltime),
\end{equation*}
where $G_N = G_N(t)$ represents the infiltration or melt rate at time $t$. Nevertheless, when the infiltration capacity of the soil is exceeded, water starts to accumulate at the surface. In this situation, the pressure-head at the soil surface reaches atmospheric pressure, and the boundary condition becomes pressure-controlled, i.e. this situation can be modelled by switching the Neumann conditions to a Dirichlet condition $\ph{} = g_D = 0  \text{ on } \Gamma_{R}^{\ph}$. These two alternating regimes can be summarized as \cite{Scudeler2017, Park2026}:
\begin{equation*}
    \ph{} \leq 0,\qquad  K(\ph)(\nabla \ph + \ee_z)  \cdot \nn \leq G_N,\quad \ph{} (K(\ph)(\nabla \ph + \ee_z)\cdot \nn - G_N) = 0, \quad \text{on } \Gamma_{R}^{\ph} \quad \forall t \in (0, \finaltime),
\end{equation*}
i.e., if the porous medium is unsaturated at the surface, the imposed flux $G_N$ infiltrates into the soil. Otherwise, if the infiltration capacity of the soil is exceeded, the pressure-head reaches atmospheric pressure $\ph{} = 0$ and the excess of water accumulates at the surface \cite{Diskin1996}. In other words, we either impose a homogeneous Dirichlet boundary condition ($\ph = g_D = 0$) or a non-homogeneous Neumann boundary condition ($K(\ph)(\nabla \ph + \ee_z) \cdot \nn = G_N$).

The easiest way to handle these kinds of conditions can be realized by manually switching between the Neumann and Dirichlet boundary conditions in time, respectively, according to the generic scheme \cite{Dolejsi2019}:
\begin{equation*}
    \begin{aligned}
        &\text{if }\qquad \ph < g_D\quad \text{and}\quad K(\ph)(\nabla \ph + \ee_z) \cdot \nn > G_N\qquad \text{then}\qquad K(\ph)(\nabla \ph + \ee_z) \cdot \nn = G_N,\\
        &\text{if }\qquad \ph < g_D\quad \text{and}\quad K(\ph)(\nabla \ph + \ee_z) \cdot \nn < G_N\qquad \text{then}\qquad K(\ph)(\nabla \ph + \ee_z) \cdot \nn = G_N,\\
        &\text{if }\qquad \ph > g_D\quad \text{and}\quad K(\ph)(\nabla \ph + \ee_z) \cdot \nn > G_N\qquad \text{then}\qquad K(\ph)(\nabla \ph + \ee_z) \cdot \nn = G_N,\\
        &\text{if }\qquad \ph > g_D\quad \text{and}\quad K(\ph)(\nabla \ph + \ee_z) \cdot \nn < G_N\qquad \text{then}\qquad \ph = g_D,\\
    \end{aligned}
\end{equation*}
for a given value of the pressure-head $g_D$ and a given flux $G_N$. However, this approach may be computationally demanding.

Thus, to effectively and automatically switch from Dirichlet to Neumann boundary conditions over $\Gamma_R^{\ph}$, and vice versa, we consider the following time-varying boundary conditions:
\begin{equation}
    K(\ph)(\nabla \ph + \ee_z) \cdot \nn + \gamma(\ph) \ph = G_N + \gamma(\ph) g_D\qquad \text{on } \Gamma_R^{\ph} \quad \forall t \in (0, \finaltime).
    \label{eq:robin_condition}
\end{equation}
We notice that in the limit $\gamma \to 0$, Equation \eqref{eq:robin_condition} tends to the pure Neumann BC of Richards' equation, while we recover the pure Dirichlet BC in the limit $\gamma \to \infty$.
The parameter $\gamma$ depends on the pressure-head and can be defined as:
\begin{equation*}
    \gamma(\ph) = \frac{r(\ph)}{1 - r(\ph)},
\end{equation*}
where 
\begin{equation*}
    r(\ph(\xx,t)) = \begin{cases}
        1 & \text{if } \ph \geq g_D\quad \text{and}\quad K(\ph)(\nabla \ph + \ee_z)\cdot \nn \leq G_N\quad (\xx,t) \in \Gamma_R^{\ph} \times (0,\finaltime),\\
        0 & \text{otherwise}.
    \end{cases}
\end{equation*}
However, if $r$ takes only the discrete values $\{0, 1\}$, then the corresponding non-linear algebraic equation often fails to converge \cite{Dolejsi2019}. To avoid this unpleasant phenomenon, the transition from $r = 0$ to $r = 1$ has to be smooth. In \cite{Dolejsi2019}, it has been proposed to choose
\begin{equation*}
    r(\ph) = r_{\mathrm{ph}}(\ph) r_{\mathrm{fl}}(\ph),
\end{equation*}
where
\begin{gather}
    r_{\mathrm{ph}}(\ph) = \begin{cases}
        0 & \text{if } \ph < g_D, \\
        \left(\frac{\ph - g_D}{s_{\mathrm{ph}}}\right)^2\left(3 - 2 \frac{\ph - g_D}{s_{\mathrm{ph}}}\right) & \text{if } \ph \in [g_D, g_D + s_{\mathrm{ph}}] ,\\
        1 & \text{if } \ph > s_{\mathrm{ph}},
    \end{cases}\\
    r_{\mathrm{fl}}(\ph) = \begin{cases}
        1 & \text{if } K(\ph)(\nabla \ph + \ee_z) \cdot \nn < G_N, \\
        1-\left(\frac{K(\ph)(\nabla \ph + \ee_z) \cdot \nn - G_N}{s_{\mathrm{fl}}}\right)^2\left(3 - 2 \frac{(K(\ph)(\nabla \ph + \ee_z) \cdot \nn - G_N)}{s_{\mathrm{fl}}}\right) & \text{if } K(\ph)(\nabla \ph + \ee_z) \cdot \nn \in [G_N, G_N + s_{\mathrm{fl}}] ,\\
        0 & \text{if } K(\ph)(\nabla \ph + \ee_z) \cdot \nn > s_{\mathrm{fl}} + G_N,
    \end{cases}
    \label{eq:nitshce:ratio_terms}
\end{gather}
$s_{\mathrm{ph}}$ and $s_{\mathrm{fl}}$ should scale as the diameter of the element. In particular, $s_{\mathrm{fl}}$ may be zero, but $s_{\mathrm{ph}}$ has to be positive.

\section{Numerical discretization}\label{sec:numerical_discretization}

In this section, we introduce the numerical discretization of problem \eqref{eq:var-problem}. In particular, here we adopt a stabilization-free virtual element method for the discretization in space alongside the Nitsche's method to handle time-varying BC and a backward Euler scheme for time discretization. Finally, a Picard scheme is adopted here to deal with the non-linearities of the Richards' equation.

\subsection{The stabilization-free virtual element method}

Let $\Th$ denote a conforming polygonal tessellation of $\Omega$ and $E\in\Th$ denote a generic polygon. Let $h_E$ denote the diameter of an element $E\in\Th$ and
$h:=\max_{E\in\Th} h_E$. 
We assume that
$\Th$ satisfies the standard virtual element mesh assumptions (see, for instance,
\cite{Beirao2017,Brenner2018}), i.e. $\exists \kappa > 0$ such that
\begin{enumerate}[label=\textbf{A.\arabic*}]
\item\label{meshA_1} for all $E\in \Th$, $E$ is star-shaped with respect to a
  ball of radius $\rho\geq \kappa h_E$;
\item\label{meshA_2} for all edges $e\subset \partial E$, the length $h_e:=|e|$ is such that $h_e \geq \kappa h_E$.
\end{enumerate}
For any given $E\in\Th$ and integer $k \geq 0$, let $\Poly{k}{E}$ be the space of polynomials of degree
up to $k$ defined on $E$.

Let $\Proj[\nabla]{1}{E}:\sob{1}{E}\rightarrow\Poly{1}{E}$ be the
$\sob{1}{E}$-orthogonal projection defined up to a constant by the
orthogonality condition:
\begin{equation}\label{eq:PiNablaorthogonalitycondition}
  \forall \,v\in\sob{1}{E}, \quad  \scal[E]{\nabla\left(\Proj[\nabla]{1}{E} v -v \right)}{\nabla p}=0 \;\; \forall \, p\in\Poly{1}{E}.
\end{equation}
To uniquely define $\Proj[\nabla]{1}{E}$, we further set $P_0(\Proj[\nabla]{1}{E}  v - v) = 0$, where
\begin{equation}
    P_0(v) \coloneq \frac{1}{\Nv} \sum_{i=1}^{\Nv} v(\bm{X}_i)\,,
    \label{eq:defP0}
\end{equation}
and $\bm{X}_i$ denotes the $i$-th vertex of $E$.

For any given $E\in\Th$, the local virtual element space of order $1$ for the pressure-head variable $\ph$ is defined as:
\begin{align}
  \Vspace[E]{1}:=\{ v_h\in \sob{1}{E} :\; &\Delta v_h\in \Poly{1}{E}, \;\;  v_h\in \Bk{1}{\partial E},\,\\
  \label{eq:enhancement}&\scal[E]{v_h}{p} =\scal[E]{\Proj[\nabla]{1}{E} v_h}{p} \; \forall p\in\Poly{1}{E}
  \}\,,
\end{align}
where $\Bk{1}{\partial E} := \{v_h\in\cont{\partial E}: v_{h|e}\in \Poly{1}{e}\;\forall e \subset \partial E\}$.

We recall that the degrees of freedom of this space are the values of functions at the vertices of $E$ (see \cite{Ahmad2013,LBe16}). Moreover, we define the global virtual element discrete space as
\begin{equation*}
  \Vspace{1} := \{v_h\in\VP \cap \con{0}{\overline{\Omega}}\colon v_{h|E} \in \Vspace[E]{1}\}\,.
\end{equation*}

Given these definitions, the local and the global vector-valued virtual element space for the displacement variable $\uu$ are defined as 
\begin{equation*}
\bsV[E]{1} = [\Vspace[E]{1}]^2, \quad \bsV{1}:= \{\bs{v}_h\in\mathbf{V}\colon \bs{v}_{h|E}\in\bsV[E]{1}\}\,.
\end{equation*}

Now, we introduce the stabilization-free scheme defined in \cite{Marcon2024}.
Given $\ell_E \in \mathbb{N}$, let
$\projhat{\ell_E}{E}\nabla: \sob{1}{E} \rightarrow \curl \Poly{\ell_E+1}{E} $
be the $\leb{2}{E}$-projection operator of the gradient of functions in
$\sob{1}{E}$, defined, $\forall v\in \sob{1}{E}$, by the orthogonality
condition
\begin{equation}
  \scal[E]{\projhat{\ell_E}{E} \nabla v}{\curl p} = \scal[E]{\nabla v}{\curl p}  \quad\forall p \in \Poly{\ell_E+1}{E}\,,
  \label{eq:PiZeroLGrad orthogonality condition}
\end{equation}
where for any $p\in \Poly{\ell_E+1}{E}$,
$ \curl p = \left(\frac{\partial p}{\partial z},- \frac{\partial p}{\partial
    x}\right)$.  Notice that for any $p\in\Poly{\ell_E+1}{E}$, $\curl p=\boldsymbol{0}$
if and only if $p$ is the constant polynomial.
\begin{remark}[Computation of $\projhat{\ell_E}{E} \nabla$]
For each function $v_h\in \Vspace[E]{1}$, the above projection is computable given the degrees of freedom of $v_h$. Indeed, $\forall p\in\Poly{\ell_E+1}{E}$ we
have, thanks to known results about De Rham diagrams in Sobolev spaces, see e.g. \cite{Boffi2013},
\begin{equation}
  \scal[E]{\nabla v_h}{\curl p}
  = \scal[\partial E]{\nabla v_h\cdot\boldsymbol{t}^{\partial E}}{p}\,,
    \label{eq:projOntheBoundary}
\end{equation}
where $\tt^{\partial E}$ denotes the unit tangent vector to the boundary of the element $E$.
\end{remark}
Moreover, for any given degree $s\geq 0$ let $\Proj{s}{E}:\leb{2}{E}\to\Poly{s}{E}$ be the $\leb{2}{}$-orthogonal projection operator, defined for any $v\in\leb{2}{E}$ such that
\begin{equation*}
    \scal[E]{\Proj{s}{E} v-v}{p} = 0 \; \forall p\in\Poly{s}{E}\,.
\end{equation*}
For any function $v_h\in\Vspace[E]{1}$, the above projector is computable for $s\leq 1$ given the degrees of freedom and the enhancement condition \eqref{eq:enhancement}. Moreover, given the definition of the projector in \eqref{eq:PiZeroLGrad orthogonality condition}, for any function $\bs{v}=(v_x,v_z)\in \vsob{1}{E}{2}$, let $\projhat{\ell_E}{E}\nabla \bs{v}$ be defined such that
\begin{equation*}
    \projhat{\ell_E}{E}\nabla \bs{v} = \left(
        \projhat{\ell_E}{E}\nabla v_x ,\; \projhat{\ell_E}{E}\nabla v_z
    \right)^\intercal,
\end{equation*}
and the discrete counterparts of $\strain (\bs{v})$ and $\div \bs{v}$ are defined as
\begin{equation*}
    \strain_{\ell_E} (\bs{v}) = \frac12 \left( \projhat{\ell_E}{E}\nabla \bs{v} + \left(\projhat{\ell_E}{E}\nabla \bs{v}\right)^T \right)\,,\quad \div_{\ell_E} \bs{v} = \tr{\strain_{\ell_E} (\bs{v})},
\end{equation*}
respectively. Moreover, the vector-valued counterpart of the $\leb{2}{}$-orthogonal projector is given by
\begin{equation*}
    \Proj{1}{E}\bs{v}_h = \left(
        \Proj{1}{E} (v_h)_x ,\; \Proj{1}{E} (v_h)_z 
    \right) ^\intercal\,.
\end{equation*}

We remark that the choice of the degree $\ell_E$ may depend on the geometry of $E$. Hence, it can be different from one polygon to another. We discuss it in more detail in the following section.

\subsection{Semi-discretization in space: The Nitsche's method for dealing with time-varying boundary conditions}\label{sec:nitshce_method}

For the space discretization of the Richards' equation, we introduce a stabilization-free virtual element method approximation based on Nitsche's method \cite{Mika2009, TagliabueQuarteroni2016} to handle the time-varying boundary conditions defined in Equation \eqref{eq:robin_condition}.

Let us denote by $\Eh[R]$ the set of edges $e$ of the tessellation belonging to $\Gamma_R^{\ph}$.
Moreover, let $\ER \in \Th$ be the boundary element having $e$ as an edge, i.e. $e = \ER \cap \Gamma_R^{\ph}$, for each $e \in \Eh[R]$.

Given $\zh \in \Vspace[E]{1}$, let $\ahE{}{}\colon \Vspace[E]{1}\times\Vspace[E]{1} \to \mathbb{R}$, $\mhE{}{}\colon \Vspace[E]{1}\times\Vspace[E]{1} \to \mathbb{R}$, and $\AhE{}{}\colon \bsV[E]{1}\times\bsV[E]{1} \to \mathbb{R}$ be defined as
\begin{gather*}
  \ahE{u_h}{v_h; \zh} \coloneq \scal[E]{K(\Proj{1}{E} \zh) \projhat{\ell_E}{E} \nabla u_h}{\projhat{\ell_E}{E} \nabla v_h},  \\
  \mhE{u_h}{v_h; \zh} \coloneq \scal[E]{C(\Proj{1}{E}\zh)\Proj{1}{E} u_h}{\Proj{1}{E} v_h},\\
  \AhE{\uu_h}{\vv_h} \coloneq \scal[E]{2\mu \strain_{\ell_E}(\uu_h)}{\strain_{\ell_E}(\vv_h)} + \scal[E]{\lambda \div_{\ell_E} \uu_h}{\div_{\ell_E} \vv_h},
\end{gather*}
for any $u_h,v_h \in \Vspace[E]{1}$ and $\uu_h,\vv_h \in \bsV[E]{1}$.

Let $\penalt : \Gamma_R^{\ph} \times (0,\finaltime) \to (0, \infty)$ be a measurable penalty function for which there exist two positive constants $\penalt_0$ and $\penalt_{\infty}$ such that
\begin{equation}
    0 < \penalt_0 \leq \penalt(\xx, t) \leq \penalt_{\infty} < \infty\quad \forall (\xx,t) \in \Gamma_{R}^{\ph}\times(0, \finaltime).
\end{equation}

Given $\zh \in \Vspace{1}$, we define $\ahE[]{\cdot}{\cdot; \zh}\colon \Vspace{1}\times\Vspace{1}\to \mathbb{R}$, $\mhE[]{\cdot}{\cdot;\zh}\colon \Vspace{1}\times\Vspace{1}\to \mathbb{R}$, $\NhE[]{\cdot}{\cdot; \zh}\colon \Vspace{1}\times\Vspace{1} \to \mathbb{R}$, $\FhE[]{\cdot; \zh}\colon \Vspace{1} \to \mathbb{R}$, $\NFhE[]{\cdot; \zh} : \Vspace{1} \to \mathbb{R}$, $\AhE[]{\cdot}{\cdot}\colon \bsV{1}\times\bsV{1} \to \mathbb{R}$, and $\bhE[]{\cdot; \zh}\colon \bsV{1}  \to \mathbb{R}$  as
\begin{align}
    \label{eq:defah} \ahE[]{u_h}{v_h; \zh} &\coloneq \sum_{E\in\Th}\ahE{u_h}{v_h; \zh},\\
     \label{eq:defmh} \mhE[]{u_h}{v_h; \zh} &\coloneq \sum_{E\in\Th}\mhE{u_h}{v_h; \zh},\\
    \nonumber \FhE[]{v_h; \zh} &\coloneq \sum_{E\in\Th}\Big[\scal[E]{f(t)}{\Proj{1}{E}v_h} - \scal[E]{K(\Proj{1}{E}\zh) \ee_z}{ \projhat{\ell_E}{E} \nabla v_h} \Big]\\
    \label{eq:defAh} \AhE[]{\uu_h}{\vv_h} &\coloneq \sum_{E\in\Th}\AhE{\bs{u}_h}{\bs{v}_h},\\
    \nonumber \bhE[]{\vv_h; \zh} &\coloneq \sum_{E\in\Th}\scal[E]{\bb(\Proj{1}{E}\zh)}{\Proj{1}{E}\bs{v}_h},
\end{align}
\begin{align}
    \label{eq:defNhE} \NhE[]{u_h}{v_h; \zh} &\coloneq 
         \displaystyle\sum_{e\in\Eh[R]} \Big[\scal[e]{\frac{-\gamma(\zh) h_e}{\penalt
        + \gamma(\zh) h_e} K(\zh) \projhat{\ell_{\ER}}{\ER}\nabla u_h \cdot \nn}{ v_h} \\
       \nonumber &\qquad+ \scal[e]{u_h}{ \frac{\penalt \gamma(\zh) }{\penalt + \gamma(\ph) h_e}  v_h} \\
       \nonumber &\qquad +\scal[e]{K(\zh) \projhat{\ell_{\ER}}{\ER}\nabla u_h \cdot \nn}{\frac{-h_e}{\penalt + \gamma(\zh) h_e}K(\zh)\projhat{\ell_{\ER}}{\ER}\nabla v_h \cdot \nn} \\
       \nonumber &\qquad + \scal[e]{u_h}{\frac{-\gamma(\zh)h_e}{\penalt + \gamma(\zh) h_e}K(\zh)\projhat{\ell_{\ER}}{\ER}\nabla v_h \cdot \nn} \Big] , 
       \end{align}
       \begin{align}
    \label{eq:defNFhE} \NFhE[]{v_h; \zh} &\coloneq
         \displaystyle\sum_{e\in\Eh[R]} \Big[-\scal[e]{K(\zh) \ee_z \cdot \nn}{\frac{-\gamma(\zh) h_e}{\penalt
        + \gamma(\zh) h_e} v_h} \\
       \nonumber &\qquad -  \scal[e]{K(\zh) \ee_z \cdot \nn}{\frac{-h_e}{\penalt + \gamma(\zh) h_e}K(\zh)\projhat{\ell_{\ER}}{\ER}\nabla v_h \cdot \nn} \\
       \nonumber &\qquad + \scal[e]{G_N}{\frac{\penalt}{\penalt + \gamma(\zh) h_e}  v_h}  + \scal[e]{g_D}{\frac{\penalt \gamma(\zh) }{\penalt + \gamma(\zh) h_e}v_h}\\
       \nonumber &\qquad +\displaystyle\sum_{e\in\Eh[R]} \Big[ \scal[e]{G_N}{\frac{-h_e}{\penalt + \gamma(\zh) h_e}K(\zh)\projhat{\ell_{\ER}}{\ER}\nabla v_h \cdot \nn} \\            
       \nonumber &\qquad + \scal[e]{g_D}{\frac{-h_e \gamma(\zh) }{\penalt + \gamma(\zh) h_e}K(\zh)\projhat{\ell_{\ER}}{\ER}\nabla v_h \cdot \nn}\Big] ,
\end{align}
for any $u_h,v_h \in \Vspace{1}$ and $\uu_h,\vv_h \in \bsV{1}$. 

The semi-discrete approximation of problem \eqref{eq:var-problem} reads as: \textit{Find $\ph_h\in\leb{2}{0,\finaltime;\Vspace{1}}$ with $\frac{\partial \ph_h}{\partial t}\in\leb{2}{0,\finaltime;\Vspace{1}}$ and $\uu_h\in \leb{2}{0,\finaltime; \bsV{1}}$ such that, for a.e. $t\in(0,\finaltime)$,}
\begin{align}
    \label{eq:space-problem}
    \begin{cases}
        \begin{aligned}
            &\mhE[]{\frac{\partial \ph_h(t)}{\partial t}}{v_h; \ph_h(t)} + \ahE[]{\ph_h (t)}{v_h; \ph_h(t)} + \NhE[]{\ph_h (t)}{v_h; \ph_h(t)} \\
            &\qquad \qquad \qquad = \FhE[]{v_h; \ph_h(t)} + \NFhE[]{v_h; \ph_h(t)}
        \end{aligned}
        & \forall v_h \in \Vspace{1},\\
        \AhE[]{\bs{u}_h(t)}{\bs{v}_h} + \bhE[]{\vv_h; \ph_h(t)} = 0
        & \forall \bs{v}_h \in \bsV{1},\\
        \ph_h(0) = \ph_{h}^0,
    \end{cases}
\end{align}
where $\ph_{h}^0$ is the nodal virtual element interpolation of the continuous initial condition $\ph_0$. 

The terms \eqref{eq:defNhE} and \eqref{eq:defNFhE} are consistency terms, while the remaining terms defining $\NhE[]{\cdot }{\cdot; \zh}$ and $\NFhE[]{\cdot; \zh}$ ensure a weak enforcement of the BCs: the larger the penalty term value $\penalt$, the more significant is the penalization on the Dirichlet data \cite{TagliabueQuarteroni2016}.

\begin{lemma}
    If the solution of \eqref{eq:richards}-\eqref{eq:momentum_balance} is such that $(\ph, \uu) \in \Poly{1}{\Omega} \times \VPoly{1}{\Omega}{2}$, then it satisfies the semi-discrete problem \eqref{eq:space-problem}, for a.e. $t\in(0,\finaltime)$.
\end{lemma}
\begin{proof}
    First, we notice that
    \begin{equation*}
        \ph = \Proj{1}{E} \ph,\quad \uu = \Proj{1}{E} \uu\quad \text{ and }\quad \nabla \ph = \projhat{\ell_E}{E} \nabla \ph,\quad \strain(\uu) = \strain_{\ell_E}(\uu),\quad \forall \ell_E \geq 0 .
    \end{equation*}
    Multiplying the Richards' equation in \eqref{eq:richards} by $v_h \in \Vspace{1}$ and the momentum equation \eqref{eq:momentum_balance} by $\vv_h \in \bsV{1}$, integrating over $E$, summing over $E \in \Th$ and using Green's formula, we obtain, for a.e. $t\in(0,\finaltime)$:
    \begin{align}
        \label{eq:lemma1_a}&\begin{aligned}
            &\mhE[]{\frac{\partial \ph}{\partial t}}{v_h; \ph} + \ah{\ph}{v_h; \ph} \\
            &\qquad \qquad \qquad- \displaystyle\sum_{e \in\ \Eh[R]}\scal[e]{K(\ph) \Big(\projhat{\ell_{\ER}}{\ER}\nabla \ph  + \ee_z\Big)\cdot \nn}{v_h}  = \FhE[]{v_h; \ph}
        \end{aligned}
        && \forall v_h \in \Vspace{1},\\
        \nonumber &\AhE[]{\bs{u}}{\bs{v}_h} + \bhE[]{\vv_h; \ph} = 0
        && \forall \bs{v}_h \in \bsV{1},\\
        \nonumber &\ph(0) = \ph_{h}^0,
    \end{align}
    by notice that virtual functions $v_h \in \Vspace{1}$ are known linear polynomial over each edge $e$ of the tessellation.
    Next, multiplying the time-varying boundary condition in \eqref{eq:richards} by $\frac{\penalt}{\penalt + \gamma(\ph) h_e} v_h$ with $v_h \in \Vspace{1}$, integrating over an edge $e \in \Eh[R]$ and summing over $e \in \Eh[R]$, we get
    \begin{equation}
    \begin{aligned}
        &\displaystyle\sum_{e\in\Eh[R]} \Big[\scal[e]{K(\ph) \Big(\projhat{\ell_{\ER}}{\ER}\nabla \ph + \ee_z \Big)\cdot \nn}{\frac{\penalt}{\penalt + \gamma(\ph) h_e} v_h} + \scal[e]{\gamma(\ph) \ph}{ \frac{\penalt}{\penalt + \gamma(\ph) h_e}  v_h} \Big] \\
        &\qquad \qquad = \displaystyle\sum_{e\in\Eh[R]} \Big[ \scal[e]{G_N}{\frac{\penalt}{\penalt + \gamma(\ph) h_e}  v_h}  + \scal[e]{\gamma(\ph) g_D}{\frac{\penalt}{\penalt + \gamma(\ph) h_e}v_h}\Big]
    \end{aligned}
    \label{eq:lemma1_b}
    \end{equation}
    Similarly, by multiplying again the time-varying boundary condition  by $\frac{-h_e}{\penalt + \gamma(\ph) h_e}K(\ph)\projhat{\ell_{\ER}}{\ER}\nabla v_h \cdot \nn$ with $v_h \in \Vspace{1}$, we obtain
    \begin{equation}
    \resizebox{0.88\textwidth}{!}{$
    \begin{aligned}
        &\displaystyle\sum_{e\in\Eh[R]} \Big[\scal[e]{K(\ph) \Big(\projhat{\ell_{\ER}}{\ER}\nabla \ph + \ee_z \Big)\cdot \nn}{\frac{-h_e}{\penalt + \gamma(\ph) h_e}K(\ph)\projhat{\ell_{\ER}}{\ER}\nabla v_h \cdot \nn} + \scal[e]{\gamma(\ph) \ph}{\frac{-h_e}{\penalt + \gamma(\ph) h_e}K(\ph)\projhat{\ell_{\ER}}{\ER}\nabla v_h \cdot \nn} \Big] \\
        &\qquad \qquad = \displaystyle\sum_{e\in\Eh[R]} \Big[ \scal[e]{G_N}{\frac{-h_e}{\penalt + \gamma(\ph) h_e}K(\ph)\projhat{\ell_{\ER}}{\ER}\nabla v_h \cdot \nn}  + \scal[e]{\gamma(\ph) g_D}{\frac{-h_e}{\penalt + \gamma(\ph) h_e}K(\ph)\projhat{\ell_{\ER}}{\ER}\nabla v_h \cdot \nn}\Big]
    \end{aligned}$}
    \label{eq:lemma1_c}
    \end{equation}
    Finally, by summing \eqref{eq:lemma1_a}, \eqref{eq:lemma1_b}, and \eqref{eq:lemma1_c}, we obtain the first equation in \eqref{eq:space-problem}.
\end{proof}

Now, we deal with the well-posedness of the semi-discrete problem \eqref{eq:space-problem}. As done in \cite{TagliabueQuarteroni2016}, for our theoretical analysis, we introduce the following $h$-dependent norm: given $\zh \in \Vspace{1}$, we define
\begin{equation}
    \norm[h]{v_h} \coloneq \left(\norm[\VP]{v_h}^ 2 + \sum_{e \in \Eh[R]}\norm[\leb{2}{e}]{\zeta_e(\zh) v_h}^2\right)^{\frac{1}{2}} \quad \forall v_h \in \Vspace{1},
\end{equation}
where
\begin{equation*}
    \zeta_e : \Vspace{1} \to (0, \infty)\quad \text{s.t.}\quad \zh \mapsto \zeta_e(\zh) = \left(\frac{\penalt \gamma(\zh)}{\penalt + h_e \gamma(\zh)} \right)^{\frac{1}{2}}.
\end{equation*}
We observe that, under mesh assumptions \ref{meshA_1}-\ref{meshA_2}, it holds
\begin{equation}
    \label{eq:weight_ineq}\vert \zeta_e(\zh(\xx, t)) \vert^2 \leq C_{\kappa} \frac{\penalt_{\infty}}{h}, \quad \forall  (\xx, t) \in \Gamma_{R}^{\ph} \times (0, \finaltime).
\end{equation}

For the ease of the reader, we report here the following fundamental result that assesses the coercivity and the continuity of the SFVEM discretized bilinear form.
First, we recall the following necessary and sufficient assumption from \cite{Marcon2024,FassinoMarcon2025}.
We assume that
\begin{enumerate}[label=\textbf{A.\arabic*}, start=3]
\item\label{meshA_3} For any $E\in\Th$, let $\ell_E$ be the smallest integer such that any polynomial $\pi_{\ell_E+1} \in \mathcal{P}_{\ell_E+1}(E)$ can be identified by a set of degrees of freedom which contains $\Nv - 1$ distinct moments
$
\frac{1}{|\partial E|}\,(\pi_{\ell_E+1} ,\xi_i)_{\partial E},
$
for a scaled polynomial basis of
$ \Poly[0]{0}{\partial E}
\coloneqq
\left\{
\xi \;:\;
\xi|_e \in \mathbb{P}_{0}(e),\ \forall e \subset \partial E,
\ \int_{\partial E} \xi = 0
\right\}.
$  {We assume that this value of $\ell_E$ exists for any polygon $E\in\Th$.}
\end{enumerate}
\begin{remark}\label{rmk:choiceOfEll}
    From a computational point of view, in \cite{Marcon2024,FassinoMarcon2025} the authors propose an algorithm
to provide the smallest value of $\ell_E$ ensuring the stability, based on a local incremental QR decomposition.
For the sake of completeness, we also remark that in \cite{Marcon2024} the sufficient condition that determines the smallest $\ell_E$ ensuring local stability is also proved theoretically on particular classes of polygons.
\end{remark}
\begin{lemma}
\label{lemma:stabfree_ineq}
Under the assumptions \ref{meshA_1}, \ref{meshA_2}, 
 \begin{equation}\label{eq:ah-continuity}
        \norm[\ell]{v_h} \coloneq \left(\displaystyle\sum_{E \in \Th}\norm[\leb{2}{E}]{\projhat{\ell_E}{E}\nabla v_h}^2\right)^{\frac{1}{2}}  \leq \norm[\VP]{v_h} \quad \forall v_h \in \Vspace{1}.
\end{equation}
Assuming also
\ref{meshA_3},  
\begin{equation*}
    \exists c_{\ast} > 0:\ \norm[\ell]{v_h} \geq c_{\ast} \norm[\VP]{v_h}\quad \forall v_h \in \Vspace{1},
\end{equation*}
where $c_{\ast}$ is independent of $h$.
\end{lemma}

In the following lemma, we prove the coercivity of the bilinear form $\ahE[]{\cdot}{\cdot;\zh} + \NhE[]{\cdot}{\cdot ;\zh}$ in terms of $\norm[h]{}$.
\begin{lemma}
    Given $\zh \in \Vspace{1}$, there exist a constant $\alpha_{\ast} > 0$, independent of $h$, such that
    \begin{equation*}
        \ahE[]{v_h}{v_h; \zh} + \NhE[]{v_h}{v_h; \zh} \geq \alpha_{\ast} \norm[h]{v_h}^2 \quad \forall v_h \in \Vspace{1}\quad \text{and a.e. in }\quad (0, \finaltime).
     \end{equation*}
\end{lemma}
\begin{proof}
To show the coercivity of the bilinear form, we first focus on its boundary terms:
\begin{align}
    \label{eq:coerc:first_term_AN}\NhE[]{v_h}{v_h; \zh} &=  \sum_{e \in \Eh[R]}\Big[-2\scal[e]{ \frac{\gamma(\zh) h_e}{\penalt + \gamma(\zh) h_e}K(\zh) \projhat{\ell_{\ER}}{\ER}\nabla v_h \cdot \nn}{v_h}  \\
    \label{eq:coerc:sec_term_AN}&\qquad+ \scal[e]{\frac{\penalt \gamma(\zh)}{\penalt + \gamma(\zh) h_e}v_h}{v_h}\\
    \label{eq:coerc:third_term_AN}&\qquad -\scal[e]{ \frac{ h_e }{\penalt + \gamma(\zh) h_e}K(\zh) \projhat{\ell_{\ER}}{\ER} \nabla v_h \cdot \nn}{K(\zh)\projhat{\ell_{\ER}}{\ER} \nabla v_h \cdot \nn}\Big].
\end{align}
The first term \eqref{eq:coerc:first_term_AN} can be bounded by using the Young inequality: $a b \leq \delta a^2 + \frac{1}{4 \delta} b^2$ for all $\delta > 0$ and $a,b \in \R$. Moreover, by applying the polynomial inverse trace inequality, with constant $C_{\rm{tr}} > 0$, and summing over $e \in \Eh[R]$, we obtain
\begin{align*}
    \sum_{e \in \Eh[e]}2\scal[e]{ \frac{\gamma(\zh) h_e}{\penalt + \gamma(\zh) h_e}K(\zh) \projhat{\ell_{\ER}}{\ER}\nabla v_h \cdot \nn}{v_h} & 
     \leq \sum_{e \in \Eh[e]}\Big[\frac{1}{\delta}K_M^2 \norm[\leb{\infty}{e}]{\frac{1}{\penalt}}^2 \norm[\leb{\infty}{e}]{\zeta_e(\zh)}^2 \left(h_e^2 \norm[\leb{2}{e}]{ \projhat{\ell_{\ER}}{\ER}\nabla v_h \cdot \nn}^2 \right)\\
     & \qquad \qquad +\delta \norm[\leb{2}{e}]{\zeta_e(\zh)v_h}^2 \Big]\\
    & \leq  \frac{1}{\delta}K_M^2 \sum_{e \in \Eh[e]} \norm[\leb{\infty}{e}]{\frac{1}{\penalt}}^2\norm[\leb{\infty}{e}]{\zeta_e(\zh)}^2 \left(h_e C_{\rm{tr}}\norm[\leb{2}{\ER}]{ \projhat{\ell_{\ER}}{\ER}\nabla v_h}^2 \right) \\
    &\qquad \quad +\delta \sum_{e \in \Eh[e]}\norm[\leb{2}{e}]{\zeta_e(\zh)v_h}^2\\
    & \leq \frac{1}{\delta}K_M^2 \frac{1}{\penalt_0^2} \norm[\leb{\infty}{\Gamma_R^{\ph}}]{\zeta_e(\zh)}^2 h C_{\rm{tr}} \norm[\leb{2}{\Omega}]{\nabla v_h}^2  +\delta \sum_{e \in \Eh[e]} \norm[\leb{2}{e}]{\zeta_e(\zh)v_h}^2.
\end{align*}
Concerning the second term \eqref{eq:coerc:sec_term_AN}, we have
\begin{align*}
    \sum_{e \in \Eh[e]}\scal[e]{\frac{\penalt \gamma(\zh)}{\penalt + \gamma(\zh) h_e}v_h}{v_h} &= \sum_{e \in \Eh[e]}\norm[\leb{2}{e}]{\zeta_e(\zh) v_h}^2.
\end{align*}

Finally, by using the polynomial discrete trace inequality and the continuity of the projector, the third term \eqref{eq:coerc:third_term_AN} can be bounded as
\begin{align*}
     \sum_{e \in \Eh[e]}\scal[e]{ \frac{ h_e }{\penalt + \gamma(\zh) h_e}K(\zh) \projhat{\ell_{\ER}}{\ER} \nabla v_h \cdot \nn}{K(\zh)\projhat{\ell_{\ER}}{\ER} \nabla v_h \cdot \nn} & \leq K_M^2 \norm[\leb{\infty}{\Gamma_R^{\ph}}]{\frac{ 1 }{\penalt + \gamma(\zh) h_e}} C_{\mathrm{tr}} \norm[\leb{2}{\Omega}]{ \nabla v_h}^2.
\end{align*}

By combining all together, by using \eqref{eq:weight_ineq}, we obtain
\begin{align*}
    \NhE[]{v_h}{v_h; \zh} &=  \sum_{e \in \Eh[R]}\Big[-2\scal[e]{ \frac{\gamma(\zh) h_e}{\penalt + \gamma(\zh) h_e}K(\zh) \projhat{\ell_{\ER}}{\ER}\nabla v_h \cdot \nn}{v_h}  + \scal[e]{\frac{\penalt \gamma(\zh)}{\penalt + \gamma(\zh) h_e}v_h}{v_h}\\
    &\qquad \quad -\scal[e]{ \frac{ h_e }{\penalt + \gamma(\zh) h_e}K(\zh) \projhat{\ell_{\ER}}{\ER} \nabla v_h \cdot \nn}{K(\zh)\projhat{\ell_{\ER}}{\ER} \nabla v_h \cdot \nn}\Big]  \\
    &\geq\Big[ -\frac{1}{\delta}K_M^2 \frac{1}{\penalt_0^2} \norm[\leb{\infty}{\Gamma_R^{\ph}}]{\zeta_e(\zh)}^2 h C_{\rm{tr}}   - K_M^2 \norm[\leb{\infty}{\Gamma_R^{\ph}}]{\frac{ 1 }{\penalt + \gamma(\zh) h_e}} C_{\mathrm{tr}} \Big] \norm[\leb{2}{\Omega}]{ \nabla v_h}^2\\
    &\qquad \quad+ (1 - \delta)\sum_{e \in \Eh[R]} \norm[\leb{2}{e}]{\zeta_e(\zh) v_h}^2 \\
    &\geq\Big[ -\frac{1}{\delta}K_M^2 \frac{\penalt_{\infty}}{\penalt_0^2} C_{\kappa}  C_{\rm{tr}}   - K_M^2 \frac{1}{\penalt_0} C_{\mathrm{tr}} \Big] \norm[\leb{2}{\Omega}]{ \nabla v_h}^2 + (1 - \delta)\sum_{e \in \Eh[R]} \norm[\leb{2}{e}]{\zeta_e(\zh) v_h}^2. \\
\end{align*}
Thus, by exploiting Lemma \ref{lemma:stabfree_ineq}, we conclude that
\begin{align*}
    \ahE[]{v_h}{v_h; \zh} + \NhE[]{v_h}{v_h; \zh} &\geq\Big[c_{\ast} -\frac{1}{\delta}K_M^2 \frac{\penalt_{\infty}}{\penalt_0^2} C_{\kappa}  C_{\rm{tr}}   - K_M^2 \frac{1}{\penalt_0} C_{\mathrm{tr}} \Big] \norm[\leb{2}{\Omega}]{ \nabla v_h}^2 + (1 - \delta)\sum_{e \in \Eh[R]} \norm[\leb{2}{e}]{\zeta_e(\zh) v_h}^2 \\
    &\geq\alpha_{\ast} \norm[h]{v_h}, \\
\end{align*}
where
\begin{equation*}
    \alpha_{\ast} = \min\Big\{c_{\ast} -\frac{1}{\delta}K_M^2 \frac{\penalt_{\infty}}{\penalt_0^2} C_{\kappa}  C_{\rm{tr}}   - K_M^2 \frac{1}{\penalt_0} C_{\mathrm{tr}}, 1 - \delta\Big\},
\end{equation*}
which is positive, provided that
\begin{equation*}
    \frac{\penalt_{\infty} C_{\kappa} C_{\rm{tr}} K_M^2}{\penalt_0 (\penalt_0 c_{\ast} - K_M^2 C_{\rm{tr}})} < \delta < 1.
\end{equation*}
This concludes the proof.
\end{proof}

Let us introduce the Lagrange basis functions $\{\varphi_i\}_{i=1}^{\Ndof[\ph]}$ w.r.t. the nodal degrees of freedom for the space $\Vspace{1}$, with $\Ndof[\ph] = \dim \Vspace{1}$, and the Lagrange basis functions $\{\vvarphi_i\}_{i=1}^{\Ndof[\uu]}$ for the space $\bsV{1}$, with $\Ndof[\uu] = \dim \bsV{1}$. Given $\zh \in \Vspace{1}$, let us introduce the following quantities, for a.e. $t \in (0, \finaltime)$
\begin{align*}
    & \mathbf{M}_{\ph}(\zh) \in \R^{\Ndof[\ph] \times \Ndof[\ph] }: && [\mathbf{M}_{\ph}(\zh)]_{ij} = \mhE[]{\varphi_i}{\varphi_j; \zh},\\
    & \mathbf{A}_{\ph}(\zh) \in \R^{\Ndof[\ph] \times \Ndof[\ph]}: && [\mathbf{A}_{\ph}(\zh)]_{ij} = \ahE[]{\varphi_i}{\varphi_j; \zh} + \NhE[]{\varphi_i}{\varphi_j; \zh},\\
    & \mathbf{A}_{\uu} \in \R^{\Ndof[\uu] \times \Ndof[\uu] }: && [\mathbf{A}_{\uu}]_{ij} = \AhE[]{\vvarphi_i}{\vvarphi_j},\\
    & \mathbf{g}_{\ph}(\zh) \in \R^{\Ndof[\ph]}: && [\mathbf{g}_{\ph}(\zh)]_{i} = -\left(\NFhE[]{\varphi_i; \zh} + \FhE[]{\varphi_i; \zh}\right), \\
    & \mathbf{g}_{\uu}(\zh) \in \R^{N^{\dof}_{\uu}}: && [\mathbf{g}_{\uu}(\zh)]_{i} =  \bhE[]{\vvarphi_i; \zh}.
\end{align*}

We rewrite \eqref{eq:space-problem} in matrix form as:
\begin{equation}
    \label{eq:matrix-space-problem}
    \begin{cases}
        \mathbf{M}_{\ph}(\ph_h(t)) \frac{d \ph_h(t)}{d t} + \mathbf{A}_{\ph}(\ph_h(t)) \ph_h(t)   + \mathbf{g}_{\ph}(\ph_h(t)) = 0,\\
        \mathbf{A}_{\uu} \uu_h(t) + \mathbf{g}_{\uu}(\ph_h(t)) = 0,\\
        \ph_h(0) = \ph_{h}^0,
    \end{cases}
\end{equation}
where we adopt the same symbols $\ph_h(t)$ and $\uu_h(t)$ to denote both the unknown functions and the vectors gathering the function degrees of freedom, denoted respectively as $\{\ph_i(t)\}_{i=1}^{\Ndof[\ph]}$ and $\{\uu_i(t)\}_{i=1}^{\Ndof[\uu]}$.

\subsection{The full discrete problem}\label{sec:full-discrete-problem}

\begin{figure}[!ht]
    \centering
   \begin{tikzpicture}[node distance=2cm]
   \node (dec1) [decision] {At step $n$};
    \node (start) [startstop, below of=dec1] {Solve Richards' equation \eqref{eq:matrix-space-time-fluid-problem} };
   \node (in1) [startstop, below of=start] {Solve momentum balance \eqref{eq:matrix-space-time-solid-problem}};
   \node (pro1) [process, below of=in1] {Evaluation of Local Factor of Safety as in \eqref{eq:discrete-LFS}};
   \node (dec2) [decision, below of=pro1] {$n = n + 1$};

  \draw [->,thick,line width=1.5pt] (start) -- (in1) node[midway, right] {\large $\ph_{h}^{n}$};
  \draw [->,thick,line width=1.5pt] (in1) -- (pro1) node[midway, right] {\large $\{\ph_{h}^{n}, \uu_h^{n}\}$};
   \draw [->,thick,line width=1.5pt] (dec1) -- (start);
   \draw [->,thick,line width=1.5pt] (pro1) -- (dec2);

   \coordinate (A) at (-4, -8);
   \coordinate (B) at (-4, -4);
   \coordinate (C) at (-4, 0);
   \draw [-,thick,line width=1.5pt] (dec2) -- (A);
   \draw [->,thick,line width=1.5pt] (A) -- (B) node[at end, left] {\large $\ph_{h}^{n-1}$};
   \draw [-,thick,line width=1.5pt] (B) -- (C);
   \draw [-,thick,line width=1.5pt] (dec1) -- (C);
   
  \end{tikzpicture}
    \caption{Workflow for solving the semi-coupled hydro-mechanical problem.}
    \label{fig:discrete-scheme}
\end{figure}
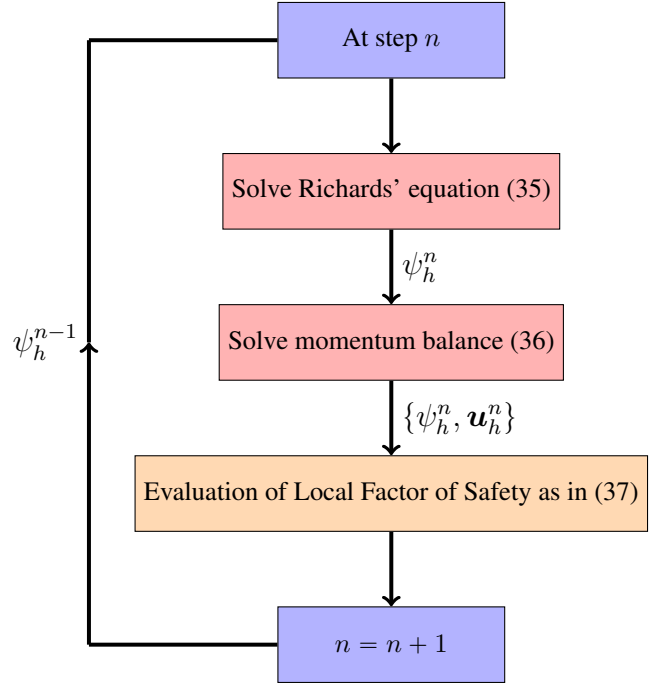

Let us describe the time semi-discretization exploiting the backward Euler scheme.
Given the time interval $[0,\finaltime]$, we subdivide it into $N_{\finaltime}$ time intervals of width $\Delta t^{n} = t^{n} - t^{n-1}$ with
\begin{equation*}
    0:= t_0 < t_1 < \dots, t_{N_{\finaltime}} := \finaltime. 
\end{equation*}
Moreover, given a function $f$ depending on time, we will denote by $f^{n} = f(t^{n})$. Given $\ph_{h}^0$, to solve the problem \eqref{eq:matrix-space-problem}, we adopt the Backward Euler scheme, i.e. for each $n = 1,\dots, N_{\finaltime}$, we solve
\begin{align}
    \label{eq:matrix-space-time-fluid-problem}
    &\mathbf{M}_{\ph}(\ph_h^{n}) \left( \ph_h^{n} - \ph^{n-1}_h \right) + \Delta t^{n}\Big[ \mathbf{A}_{\ph}(\ph_h^{n}) \ph_h^{n}   + \mathbf{g}_{\ph}(\ph_h^{n}) \Big] = 0,\\
    \label{eq:matrix-space-time-solid-problem}&\mathbf{A}_{\uu} \uu_h^{n} + \mathbf{g}_{\uu}(\ph_h^{n}) = 0.
\end{align}

First, we observe that the problem \eqref{eq:matrix-space-time-fluid-problem}-\eqref{eq:matrix-space-time-solid-problem} is semi-coupled. Consequently, following \cite{Lu2012, Moradi2018}, we can first solve the non-linear problem \eqref{eq:matrix-space-time-fluid-problem} in the variable $\ph_h^{n}$, and then solve the linear problem \eqref{eq:matrix-space-time-solid-problem} in the unknown $\uu_h^n$. Finally, given these approximations, we can proceed to compute the Local Factor of Safety \eqref{eq:LFS} as follows: for any $\xx$ in the interior of $E$,
\begin{gather}
    \nonumber \stress_h(\xx) \coloneq 2 \mu \strain_{\ell_E}(\uu_h^{n}(\xx)) + \lambda \div_{\ell_E} (\uu_h^{n}(\xx)) \II,\\
    \label{eq:discrete-LFS}{\rm{LFS}}_h(\xx) = \frac{\cos(\phi') (c' +  \sigma'_{h,I}(\xx) \tan(\phi'))}{\sigma_{h,II}(\xx)'},
\end{gather}
where 
\begin{gather*}
    \sigma_{h,I}'(\xx) = \frac{\sigma_{h,1}(\xx)  + \sigma_{h,3}(\xx)}{2}  - \chi(\Proj{1}{E} \ph_h^n(\xx)) \rho_w g \Proj{1}{E} \ph_h^n(\xx),\\
        \sigma_{h,II}'(\xx) = \frac{\sigma_{h,1}(\xx)  - \sigma_{h,3}(\xx)}{2}  ,
\end{gather*}
$\sigma_{h,1}(\xx)$ and $\sigma_{h,3}(\xx)$ are the major and minor principal of the total stress $\stress_h(\xx)$. 
The full procedure is summarized in Figure \ref{fig:discrete-scheme}.

\subsection{The linearization of Richards' equation}

To obtain the solution of the non-linear discrete Richards' equation \eqref{eq:matrix-space-time-fluid-problem}, an iterative linearization scheme must be taken into account. Following \cite{Celia1990, List2016, STOKKE2023}, the Picard scheme is employed. 


First, we define the non-linear algebraic system to be solved at time step $n=1,\dots,N_{\finaltime}$ and at each non-linear iteration $m \geq 0$:
\begin{equation*}
    \mathbf{F}(\ph_h^{n,m}) = \mathbf{M}_{\ph}(\ph_h^{n,m}) \Big[ \ph_h^{n,m}  - \ph^{n-1}_h \Big] + \Delta t^{n} \Big[\mathbf{A}_{\ph}(\ph_h^{n,m}) \ph_h^{n,m}   + \mathbf{g}_{\ph}(\ph_h^{n,m}) ] = 0.\\
\end{equation*}
Using the first-order Taylor expansion about $\ph_h^{n,m}$ for $\mathbf{F}(\ph_h^{n,m}) = 0$ yields at
\begin{equation*}
    \mathbf{F}(\ph_h^{n,m}) + \mathbf{DF}(\ph_h^{n,m}) \delta_h^{n,m} = 0,\\
\end{equation*}
where in the Picard method $\mathbf{DF}(\ph_h^{n,m}) \in \R^{\Ndof[\ph] \times \Ndof[\ph]}$ reads as:
\begin{equation}
   \mathbf{DF}(\ph_h^{n,m}) = \mathbf{M}_{\ph}(\ph_h^{n,m}) + \Delta t^{n} \mathbf{A}_{\ph}(\ph_h^{n,m}).
   \label{eq:nonlinear-global-matrix}
\end{equation}
Thus, setting $\ph_h^{n,0} = \ph_h^{n-1}$, the Picard scheme can be summarized as follows: for each $m \geq 0$
\begin{align}
    \label{eq:picard-solve}&\text{Solve }\quad &&\hspace{-80pt}\mathbf{DF}(\ph_h^{n,m}) \delta_h^{n,m} = - \mathbf{F}(\ph_h^{n,m}),\\
    \label{eq:picard-compute}&\text{Compute }\quad &&\hspace{-80pt}\ph_h^{n,m + 1} = \delta_h^{n,m} + \ph_h^{n,m},
\end{align}
until the following stopping criteria are satisfied
\begin{equation}
    \norm[2]{\mathbf{F}(\ph_h^{n,m})} \leq \text{tol}_{\text{rr}} \norm[2]{\mathbf{F}(\ph_h^{n,0})} + \text{tol}_{\text{ar}}\quad \text{and}\quad \norm[2]{\delta_h^{n,m}} \leq \text{tol}_{\text{rd}} \norm[2]{\ph_h^{n,0}} + \text{tol}_{\text{ad}},
\end{equation}
for given user-defined tolerance values $\text{tol}_{\text{rr}},\ \text{tol}_{\text{ar}},\ \text{tol}_{\text{ad}}$ and $\text{tol}_{\text{rd}}$, where $\norm[2]{}$ denotes the euclidean norm.

We observe that to solve the discrete Richards' equation with the Picard method, the matrix $\mathbf{DF}(\ph_h^{n,m})$, defined in \eqref{eq:nonlinear-global-matrix}, must be nonsingular. In the previous section, we have shown that the self-stabilized matrix $\mathbf{A}_{\ph}(\ph_h^{n,m})$ is nonsingular. On the other hand, the matrix $\mathbf{M}_{\ph}(\ph^{n,m})$ accounts only for a consistent term \cite{LBe13, LBe16}, and it is not self-stabilized. As observed in \cite{Vacca2015}, due to the presence of $\mathbf{A}_{\ph}(\ph^{n,m})$, the matrix $\mathbf{M}_{\ph}(\ph^{n,m})$ does not really need to be strictly definite positive to obtain a nonsingular matrix $\mathbf{DF}(\ph_h^{n,m})$. Nonetheless, for very small values of $\Delta t^n$, the conditioning of $\mathbf{DF}(\ph_h^{n,m})$ may become very large by not stabilizing the bilinear form $\mhE{}{}$.

A possible way to avoid this issue is introducing a stabilized bilinear form $\widetilde{m}_h^E$, which reads as \cite{Vacca2015}:
\begin{equation}
    \widetilde{m}_h^E(u_h, v_h; \zh) \coloneq \scal[E]{C(\Proj{1}{E} \zh) \Proj{1}{E} u_h}{\Proj{1}{E} v_h} + \beta^E(\zh) \stab[E]{(I - \Proj{1}{E})u_h}{(I - \Proj{1}{E})v_h},
\end{equation}
where $\stab[E]{}{} : \Vspace[E]{1} \times \Vspace[E]{1} \to \R$ is a symmetric definite positive bilinear form that must scale as the mass continuous bilinear form $\scal[E]{}{}$ and $\beta^E \in \R^+$ is a positive constant that depends on $\zh$ and must be tuned to avoid stability issues \cite{PintoreTeora2025}.

To overcome ill-conditioning and the problem-dependent stabilization term, here, instead, we consider the lumped version of the matrix $\mathbf{M}_{\ph}(\zh)$, which is a diagonal matrix whose diagonal entries are defined as:
\begin{align*}
    [\overline{\mathbf{M}}_{\ph}(\zh)]_{ii} &=  \sum_{j=1}^{\Ndof[\ph]} \sum_{E \in \Th}  \widetilde{m}_h^E(\varphi_i, \varphi_j; \zh) =  \sum_{E \in \Th}  \widetilde{m}_h^E(\varphi_i, \sum_{j=1}^{\Ndof[\ph]} \varphi_j; \zh) = \sum_{E \in \Th}  \widetilde{m}_h^E(\varphi_i, 1; \zh), 
\end{align*}
due to partition-of-unity property \cite{Paulino2020, Enabe2025}.
Since $\stab[E]{}{}$ must scale as $\scal[E]{}{}$ and $\Proj{1}{E}$ is $\leb{2}{E}$-orthogonal, we have 
\begin{align*}
    [\overline{\mathbf{M}}_{\ph} (\zh)]_{ii} &= \sum_{E \in \Th} \scal[E]{C(\Proj{1}{E} \zh)\Proj{1}{E} \varphi_i}{1} + \beta^E(\zh) \stab[E]{(I - \Proj{1}{E})\varphi_i}{1} = \sum_{E \in \Th} \int_E C(\Proj{1}{E} \zh) \Proj[\nabla]{1}{E} \varphi_i.
\end{align*}

Using such a definition, we rewrite 
\begin{gather*}
    \mathbf{F}(\ph_h^{n,m}) = \overline{\mathbf{M}}_{\ph}(\ph_h^{n,m})(\ph_h^{n,m}  - \ph^{n-1}_h) + \Delta t^{n} \Big[\mathbf{A}_{\ph}(\ph_h^{n,m}) \ph_h^{n,m}   + \mathbf{g}_{\ph}(\ph_h^{n,m}) \Big] = 0,\\
   \mathbf{DF}(\ph_h^{n,m}) = \overline{\mathbf{M}}_{\ph}(\ph_h^{n,m}) + \Delta t^{n} \mathbf{A}_{\ph}(\ph_h^{n,m}),
\end{gather*}
and solve \eqref{eq:picard-solve}-\eqref{eq:picard-compute} with these new introduced definitions.

\begin{remark}
    We recall that the row sum technique does not guarantee positive nodal masses, even with the linear polynomial space in the VEM \cite{Paulino2020}. A positive nodal mass matrix can instead be obtained by locally defining
    \begin{equation*}
        [\overline{\mathbf{M}}_{\ph}^E (\zh)]_{ii} =  \frac{\vert E \vert }{\tr \mathbf{M}^E_{\ph} (\zh)} [\mathbf{M}^E_{\ph} (\zh)]_{ii}\quad \forall i = 1,\dots ,N^v_E,\quad \forall E \in \Th,
    \end{equation*}
    where $\tr \mathbf{M}_{\ph}^E (\zh)$ denotes the trace of the elemental matrix $\mathbf{M}^E_{\ph} (\zh)$.
\end{remark}

\section{Numerical experiments}\label{sec:numerical_experiments}

In this section, we present four numerical experiments designed to assess the performance of the proposed method.

The first two experiments evaluate the performance of the stabilization-free method for the solution of the Richards' equation \eqref{eq:richards} and the elastic problem \eqref{eq:momentum_balance}, respectively. The Richards problem corresponds to the benchmark problem ``Test 1'' presented in \cite{Mitra2019}, while the elastic problem is based on the benchmark problem ``Test2b'' from \cite{Artioli2017}.

For both experiments, we consider two families of four computational meshes each: a non-uniform quadrilateral mesh, reproducing mesh ``a'' in ``Test2b'' from \cite{Artioli2017}, and a centroid-based Voronoi tessellation, corresponding to mesh ``d'' in the same reference. The finest mesh of each family is displayed in Figure \ref{fig:test1-2:mesh}.

\begin{figure}[!ht]
    \centering
    \begin{subfigure}{0.35\textwidth}
        \includegraphics[width=1\linewidth]{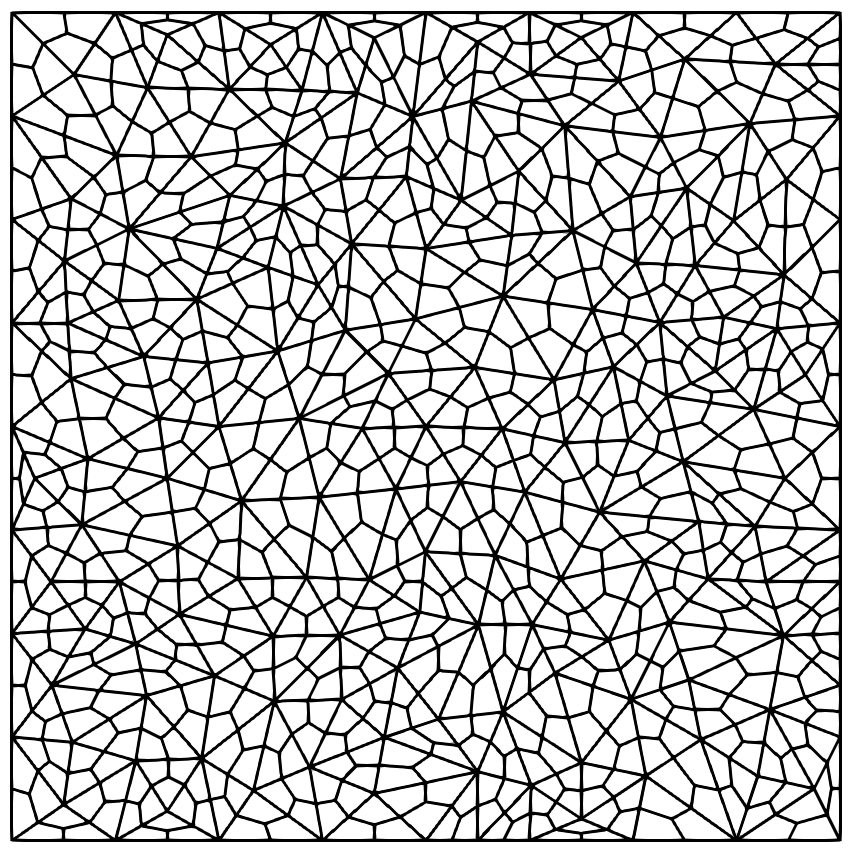}
        \caption{}
    \end{subfigure}
    \begin{subfigure}{0.35\textwidth}
        \includegraphics[width=1\linewidth]{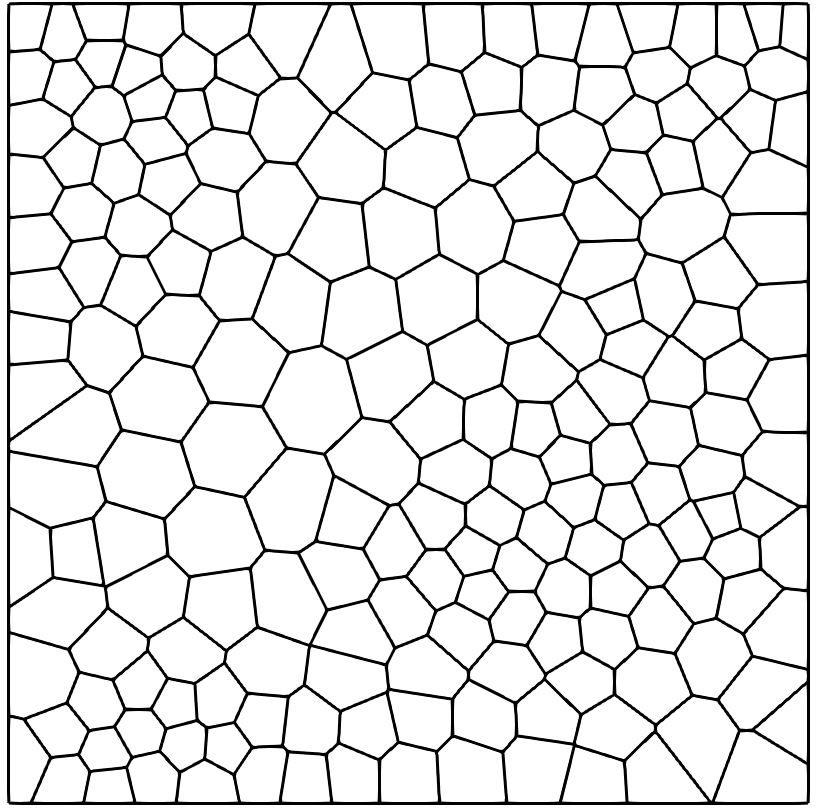}
        \caption{}
    \end{subfigure}
    \caption{Test 1, 2: Last mesh refinement. Left: Non-uniform quadrilateral. Right: Voronoi.}
    \label{fig:test1-2:mesh}
\end{figure}

The third experiment simulates an infiltration scenario, intending to explore the reliability of Nitsche's method.

Finally, in the last experiment, we simulate the hydro-mechanical model to assess the ability of the method to determine soil stability conditions.

All the numerical experiments are carried out using the PolyDiM library \cite{Polydim}.

\subsection{Test 1: Convergence test for Richards' equation}

This benchmark problem follows the ``Test 1'' presented in \cite{Mitra2019} and aims to show the performance of the stabilization-free method applied to the Richards' equation \eqref{eq:richards}. More precisely, we consider $\Omega = (0,1)^2$ and we set $\finaltime = 1$. The water content and the permeability coefficients are defined as
\begin{equation*}
    \wc(\ph) = \begin{cases}
        \frac{1}{(1-\ph)^{1 / 3}} & \text{if } \ph < 0,\\
        1 &\text{if } \ph \geq 0,
    \end{cases}\qquad K(\wc(\ph)) = \wc(\ph)^3 = \begin{cases}
        \frac{1}{(1-\ph)} & \text{if } \ph < 0,\\
        1 &\text{if } \ph \geq 0.
    \end{cases}
\end{equation*}
The forcing term and non-homogeneous Dirichlet boundary conditions are imposed in accordance with the following analytical solution
\begin{equation*}
    \ph(x, z, t) = 1-(1+t^2)(1+x^2+z^2).
\end{equation*}

\begin{figure}[!ht]
    \centering
    \includegraphics[width=0.5\linewidth]{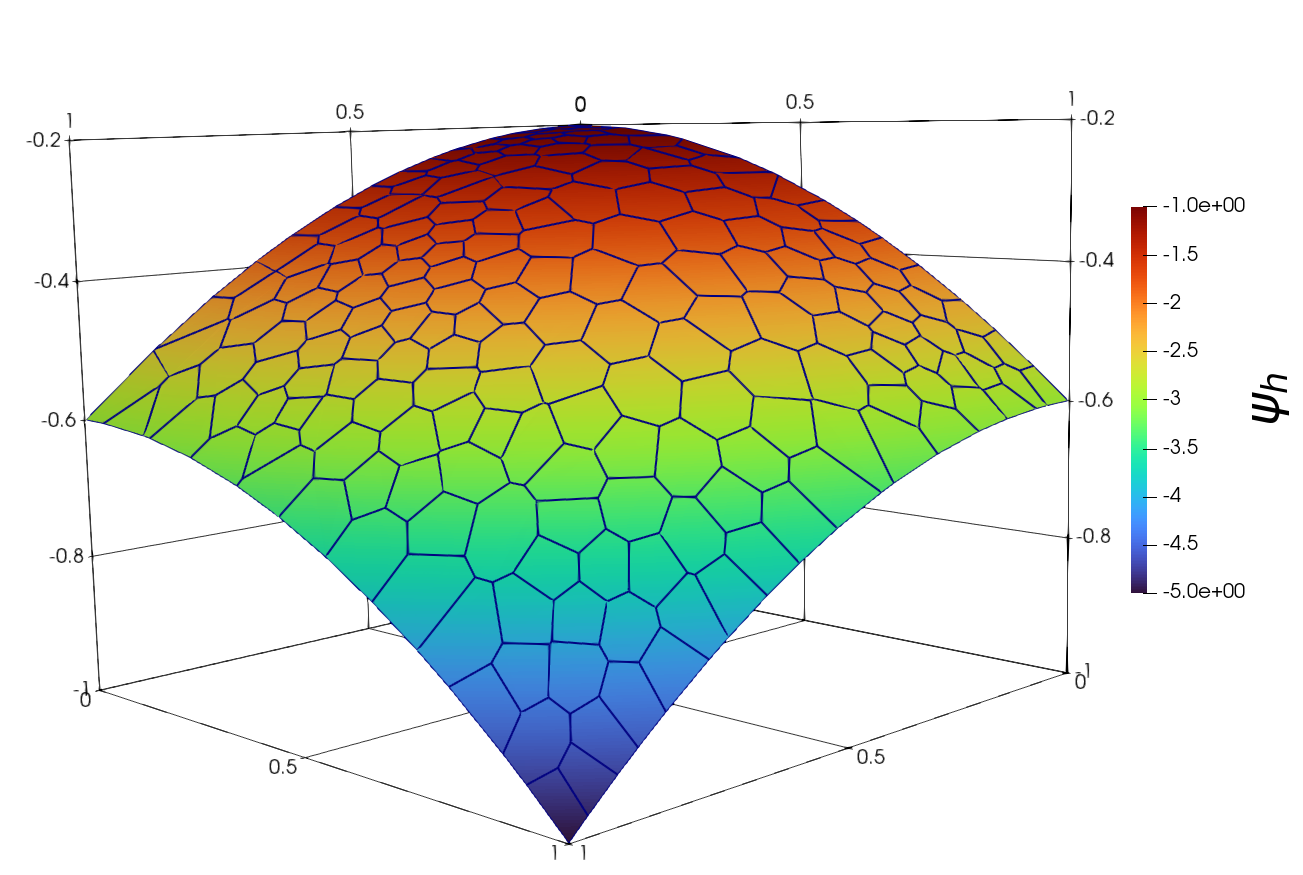}
    \caption{Test 1: Discrete solution computed at the final time over the last refinement of the Voronoi family. The solution in the image is rescaled by a factor of 0.2.}
    \label{fig:test1:solution}
\end{figure}

\begin{figure}[!ht]
    \centering
    \begin{subfigure}{0.35\textwidth}
        \includegraphics[width=1\linewidth]{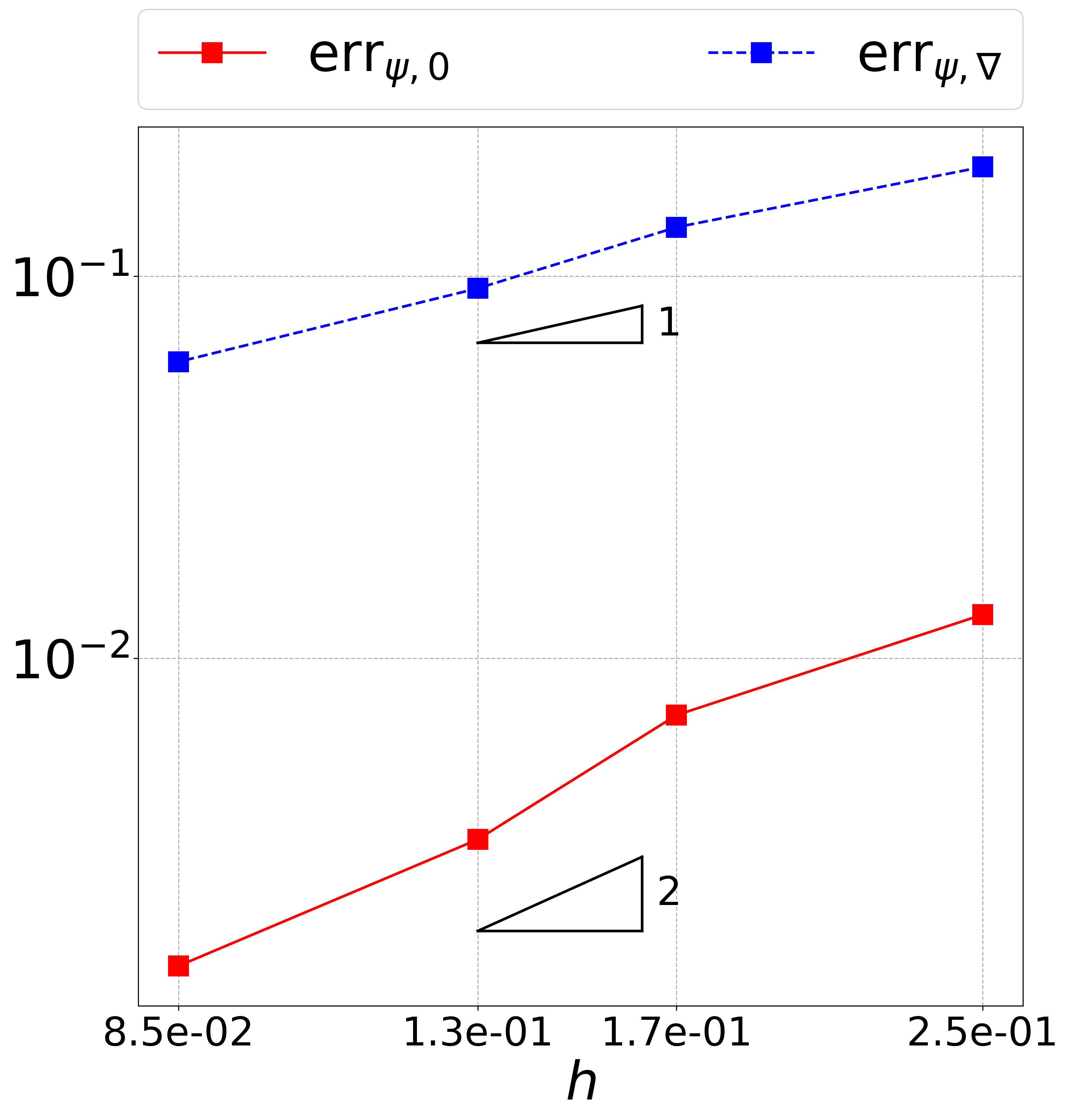}
        \caption{}
    \end{subfigure}
    \begin{subfigure}{0.35\textwidth}
        \includegraphics[width=1\linewidth]{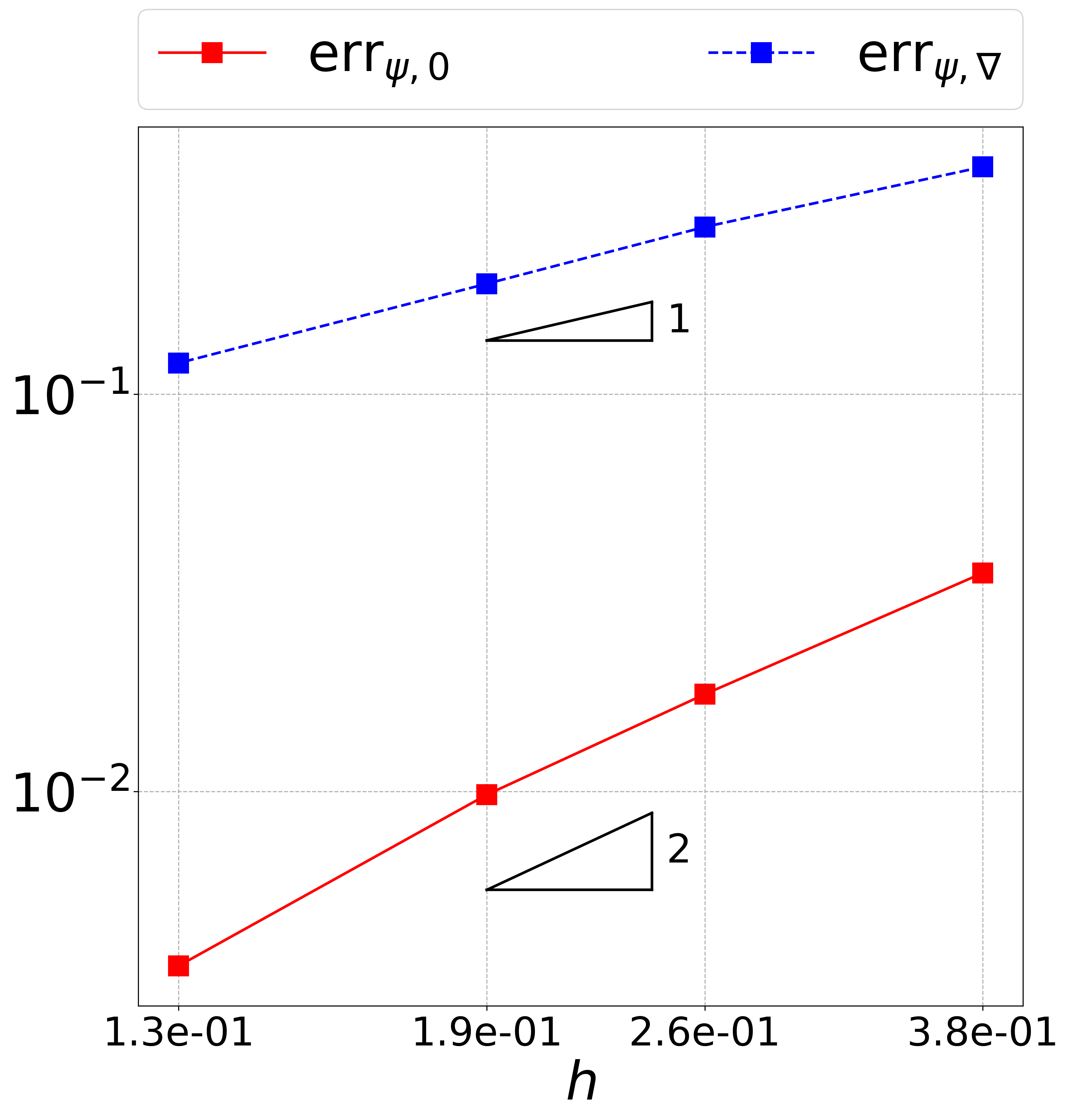}
        \caption{}
    \end{subfigure}
    \caption{Test 1: Behaviour of errors \eqref{eq:richards:errors} as $h$ decreases. Left: Non-uniform quadrilateral. Right: Voronoi.}
    \label{fig:test1:errors}
\end{figure}

To assess the performance of the method introduced in Section \ref{sec:full-discrete-problem}, we evaluate the accuracy of the approximation of the pressure variable $\ph$ by comparing the numerical solution with the exact one at the final time. More precisely, we consider the following discrete errors
\begin{equation}
    \mathrm{err}_{\ph,0}^2 = \sum_{E \in \Th} \norm[\leb{2}{E}]{\Proj{1}{E} \ph_h^{N_{\finaltime}} - \ph}^2 \quad \mathrm{err}_{\ph,\nabla}^2 = \sum_{E \in \Th} \norm[\leb{2}{E}]{\nabla \Proj[\nabla]{1}{E} \ph_h^{N_{\finaltime}} - \nabla \ph}^2,
    \label{eq:richards:errors}
\end{equation}
which measure the errors in the $\leb{2}{}$-norm and in the $H^1$-seminorm, respectively.

The convergence behaviour of these errors is investigated by decreasing the mesh size $h$ and for each considered family of mesh introduced in the previous section. In all the simulations, the time discretization is kept fixed by setting $\Delta t^n=\Delta t=0.01,\ \forall\, n=1,\dots,N_{\finaltime}$
so that the influence of the spatial discretization can be isolated. The errors \eqref{eq:richards:errors} are computed at the final time $\finaltime=1$.

The numerical approximation obtained on the finest mesh of the Voronoi family at the final time is reported in Figure~\ref{fig:test1:solution}, while the corresponding convergence curves of the errors as the mesh size $h$ decreases are displayed in Figure~\ref{fig:test1:errors}. For both mesh families, the numerical results clearly exhibit the expected convergence rates in both the $\leb{2}{}$-norm and the $H^1$-seminorm, thereby confirming the optimal polynomial accuracy of the proposed scheme.

In all the numerical experiments, the stopping criteria for the non-linear iterative solver are chosen as
\begin{equation*}
    \text{tol}_{\mathrm{rr}} = 10^{-10},
    \qquad
    \text{tol}_{\mathrm{ar}} = 10^{-10},
    \qquad
    \text{tol}_{\mathrm{ad}} = 10^{-8},
    \qquad
    \text{tol}_{\mathrm{rd}} = 10^{-8}.
\end{equation*}
With these tolerances, the iterative procedure requires, on average, approximately $8$ iterations to converge at each time step, with no significant dependence on the considered mesh family. This indicates a robust and stable behaviour of the proposed algorithm across different mesh geometries.

\subsection{Test 2: Convergence test for linear elastic problem}

In this test case, we consider the plane strain convergence test corresponding to ``Test2b'' in \cite{Artioli2017}. More precisely, we consider the following boundary value problem with homogeneous Dirichlet boundary conditions:
\begin{equation}
    \begin{cases}
        \nabla \cdot \stress(\uu) + \bb = \bm{0} & \text{in } \Omega,\\
        \stress(\uu) = 2 \mu \strain(\uu) + \lambda \div \uu \II & \text{in } \Omega,\\
        \strain(\uu) = \frac{1}{2} (\nabla \uu + (\nabla \uu)^T)& \text{in } \Omega,\\
        \uu = \bm{0} & \text{on } \Gamma,
    \end{cases}
    \label{eq:pure-elastic}
\end{equation}
where we set $\Omega = (0,1)^2$ and $\lambda = \mu = 1$. In this test, the analytical solution is defined as
\begin{equation*}
    \uu_x = \uu_z = \sin(\pi x) \sin(\pi z),
\end{equation*}
and the load term $\bb$ is derived accordingly to $\uu$. 

\begin{figure}[!ht]
    \centering
    \begin{subfigure}{0.35\textwidth}
        \includegraphics[width=1\linewidth]{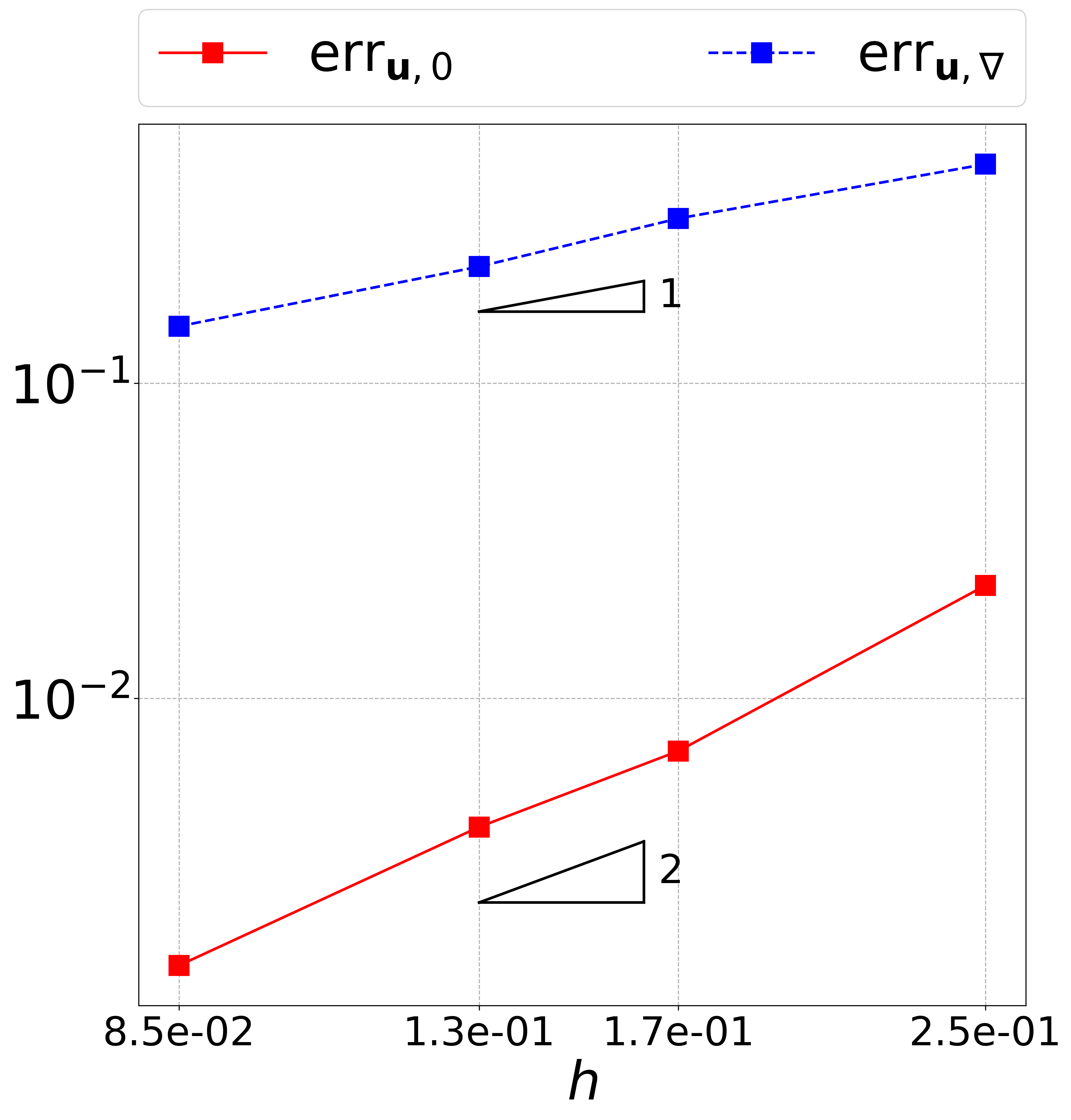}
        \caption{}
    \end{subfigure}
    \begin{subfigure}{0.35\textwidth}
        \includegraphics[width=1\linewidth]{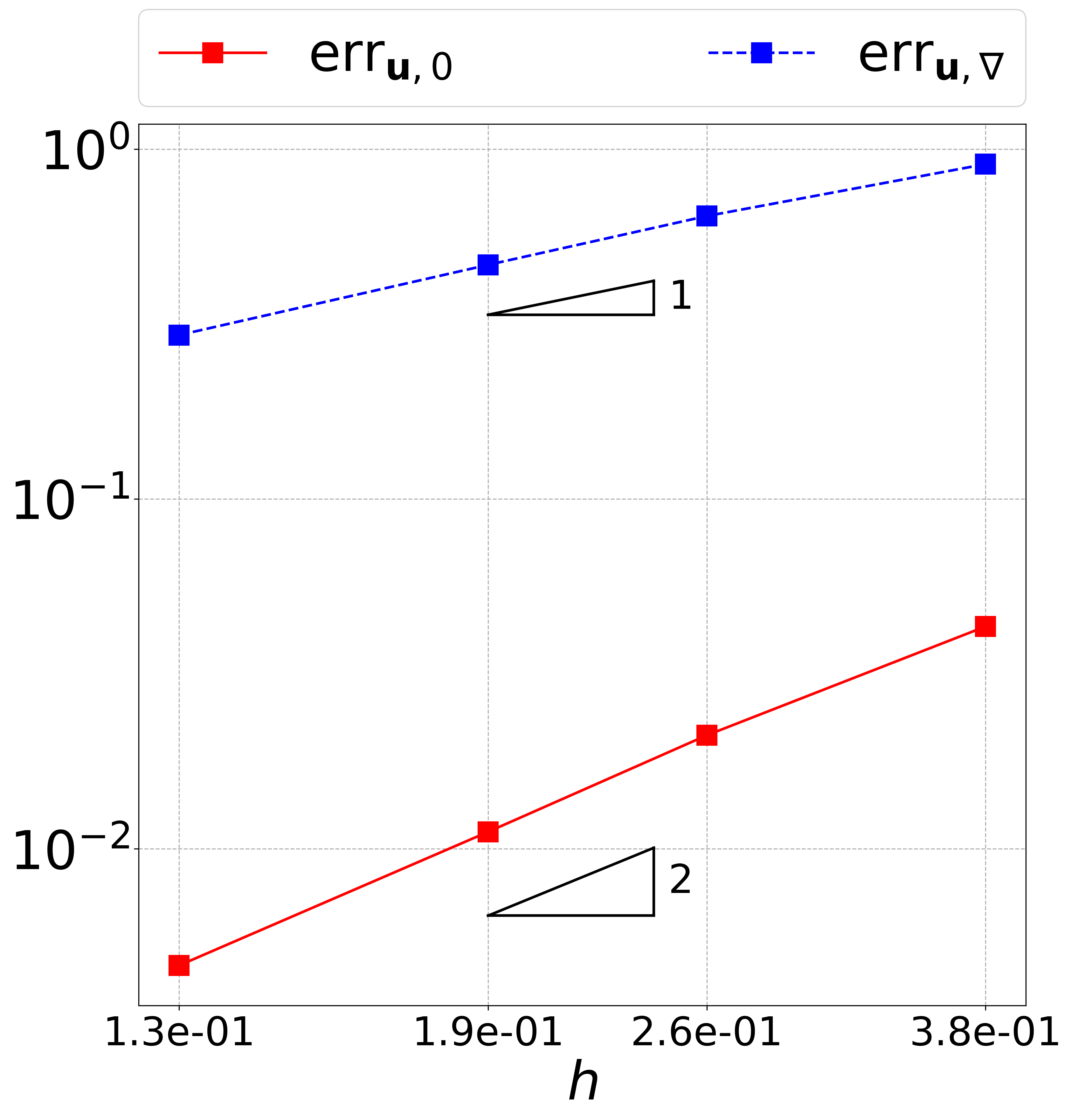}
        \caption{}
    \end{subfigure}
    \caption{Test 2: Behaviour of errors \eqref{eq:elastic:errors} as $h$ decreases. Left: Non-uniform quadrilateral. Right: Voronoi.}
    \label{fig:test2:errors}
\end{figure}

To assess the performance of the method introduced in Section \ref{sec:full-discrete-problem}, we evaluate the errors in the displacement field $\uu$ with respect to the exact solution. In particular, we consider the discrete $\leb{2}{}$- and $\sob{1}{}$-errors defined as
\begin{equation}
    \mathrm{err}_{\uu,0}^2 = \sum_{E \in \Th} \norm[\leb{2}{E}]{\Proj{1}{E} \uu_h - \uu}^2 \quad \mathrm{err}_{\uu,\nabla}^2 = \sum_{E \in \Th} \norm[\leb{2}{E}]{\nabla \Proj[\nabla]{1}{E} \uu_h - \nabla \uu}^2,
    \label{eq:elastic:errors}
\end{equation}
and study their behavior as the mesh size $h$ decreases for each of the considered mesh families: the non-uniform quadrilateral and the Voronoi families, shown in Figure \ref{fig:test1-2:mesh}.

The computed errors are reported in Figure \ref{fig:test2:errors}. Numerical results confirm that the tested method presents the expected convergence rates. More precisely, the $\sob{1}{}$-error exhibits a linear convergence rate, while the $\leb{2}{}$-error converges with quadratic accuracy, respectively. These optimal polynomial rates are achieved for both the considered families of meshes, demonstrating that the method is robust with respect to the choice of the meshes, regardless of mesh non-uniformity.

\subsection{Test 3: An infiltration process}

This test mimics the numerical example ``4.1'' in \cite{Gatti2024} and aims to assess the numerical performance of Nitsche's method when considering an infiltration process. 

The initial condition is set equal to the hydrostatic pressure $\ph_0(x, z) = - z$ and the hydraulic conductivity and the water content are here expressed by the Mualem-van Genuchten laws \eqref{eq:van:parameters} with parameters given in Table \ref{tab:test3:par}.

\begin{table}[!ht]
\centering
\caption{Test 3: Soil model parameters for the Mualem-van Genuchten laws \eqref{eq:van:parameters}.}
\label{tab:test3:par}
\begin{tabular}{@{}clll@{}}
\toprule
\textbf{Symbol} & \textbf{Parameter Name}          & \multicolumn{1}{c}{\textbf{Units}} & \multicolumn{1}{c}{\textbf{Value}} \\ \midrule
$\wc_s$         & Saturated water content           & $-$                                & 0.4                              \\
$\wc_r$         & Residual water content           & $-$                                & 0.04                              \\
$K_s$           & Saturated hydraulic conductivity & $\rm{m}\ \rm{s}^{-1}$                          & 1.0e-6                            \\
$\alpha$        & Mualem-van Genuchten fitting parameter  & $\rm{m}^{-1}$                           & 0.2                               \\
$n$             & Mualem-van Genuchten fitting parameter  & $-$                                & 1.5                              \\
\bottomrule
\end{tabular}
\end{table}

Let us consider a simple rectangular domain $\Omega = (0, 0.1) \times (0, 5)\ \rm{m}^2$. We set homogeneous Dirichlet boundary conditions at the bottom of the domain and we apply no-flux boundary conditions on the lateral boundary. 
At the top boundary, we impose a constant and uniform rainfall condition. The precipitation is assumed here to be aligned with the vertical direction, i.e.
\begin{equation}
    \bm{G}_N = - G_N \ee_z,
    \label{test3:neuman_condition}
\end{equation}
where $G_N$, as expressed in Section \ref{sec:seepage}, represents the actual rainfall rate.

In the following, we consider three different rainfall scenarios by changing the value of the ratio between rainfall and saturated permeability $G_N / K_s$,  namely 
\begin{enumerate}[label=\textbf{Case} \arabic*]
    \item \label{test3:case1}$G_N / K_s = 0.1 < 1$. The final time is set to $\finaltime = 100 \rm{h}$.
    \item \label{test3:case2}$G_N / K_s = 1$. The final time is set to $\finaltime = 50 \rm{h}$.
    \item \label{test3:case3}$G_N / K_s = 10 > 1$. The final time is set to $\finaltime = 50 \rm{h}$.
\end{enumerate}

Concerning the numerical discretization, for each tested case, we choose a uniform Cartesian mesh with edge length equal to $0.05 \rm{m}$ and a uniform time discretization with $N_{\finaltime} = 10$.

\begin{figure}[!ht]
    \centering
    \includegraphics[width=0.5\linewidth]{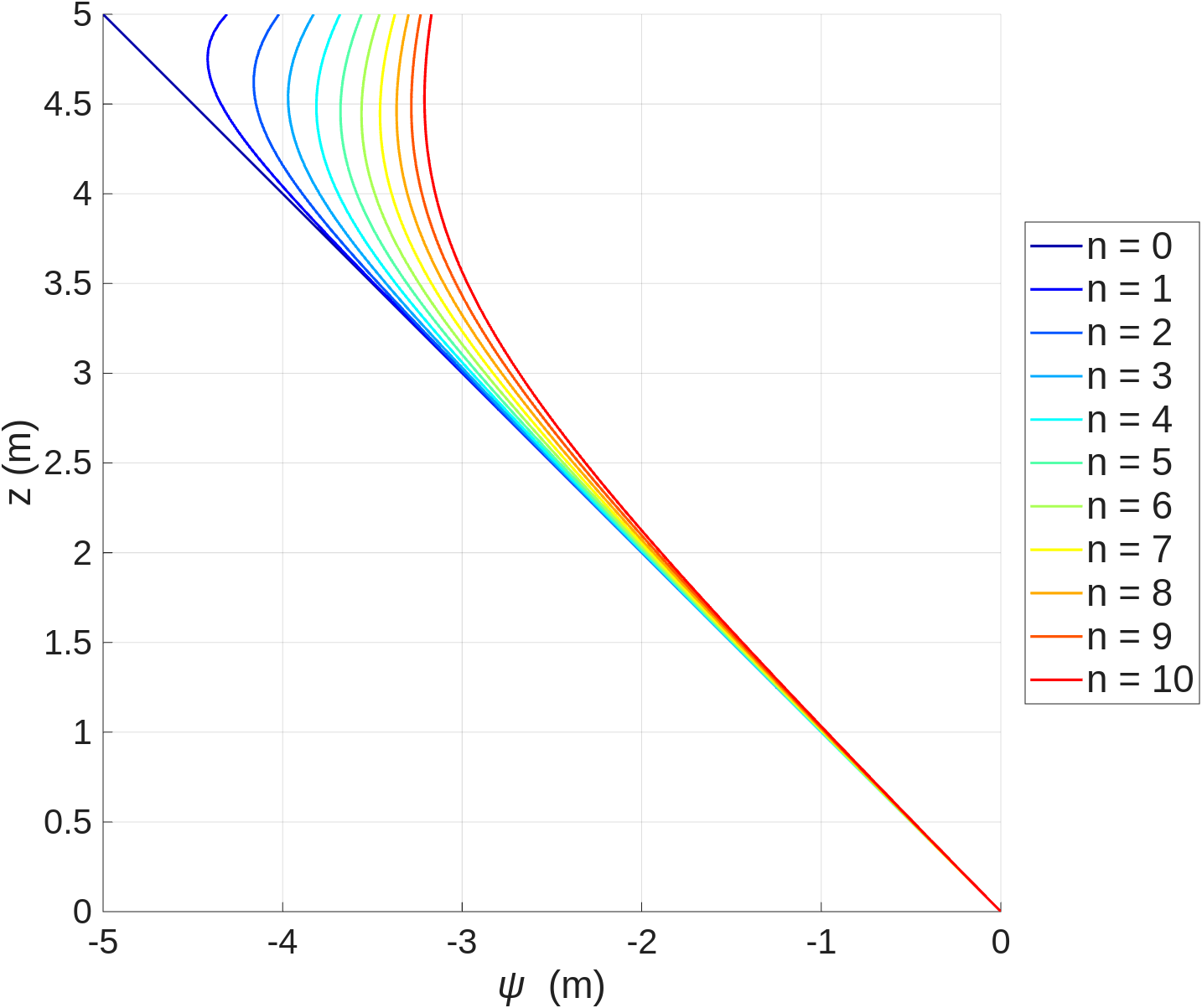}
    \caption{Test 3: Numerical solution extracted along the vertical line $x = 0.05$ for each $n = 0, \dots, N_{\finaltime}= 10$. \ref{test3:case1}.}
    \label{fig:test3:case1}
\end{figure}

\begin{figure}[!ht]
    \centering
    \includegraphics[width=0.5\linewidth]{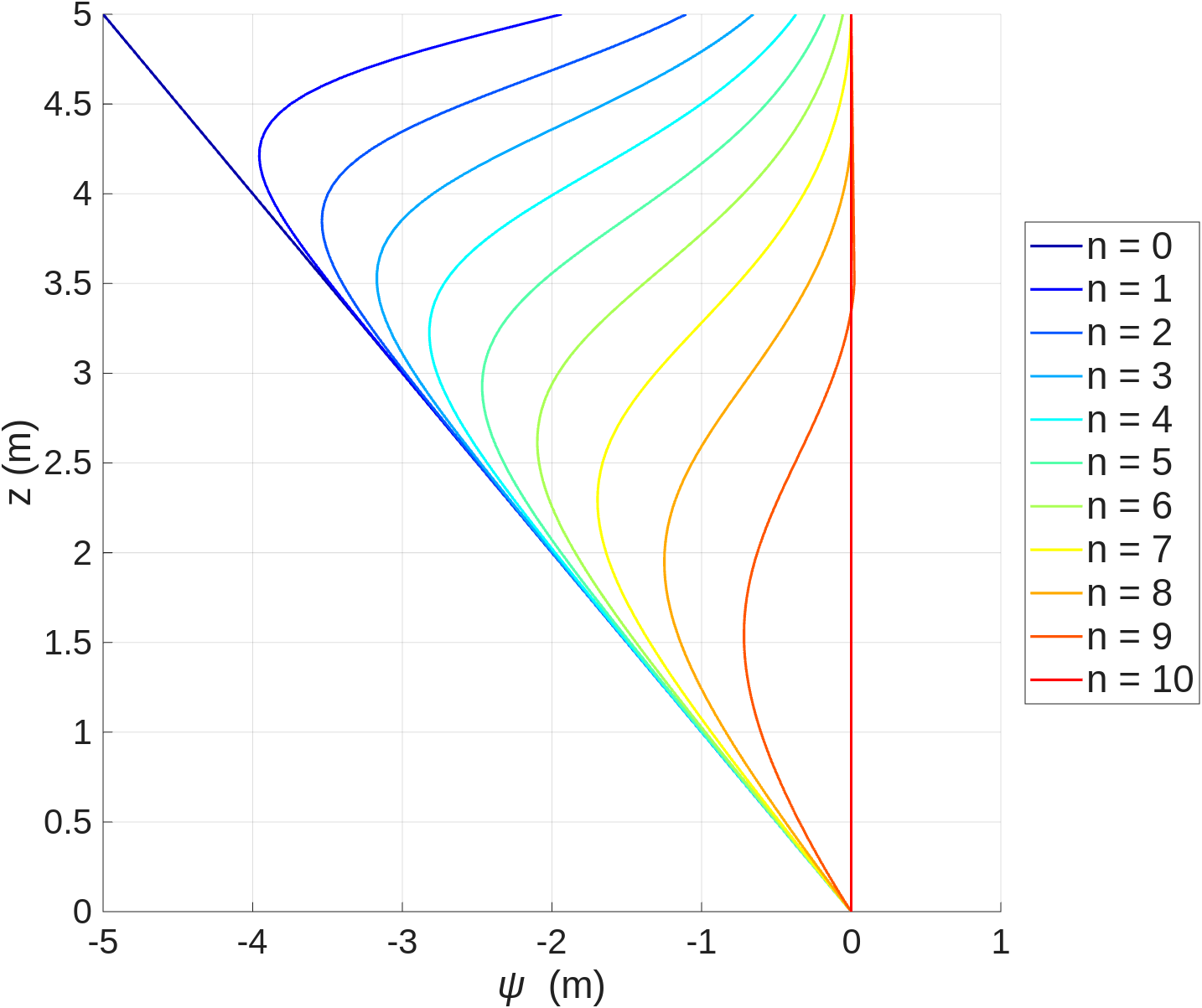}
    \caption{Test 3: Numerical solution extracted along the vertical line $x = 0.05$ for each $n = 0, \dots, N_{\finaltime} = 10$. \ref{test3:case2}.}
    \label{fig:test3:case2}
\end{figure}

\begin{figure}[!ht]
    \centering
    \includegraphics[width=0.5\linewidth]{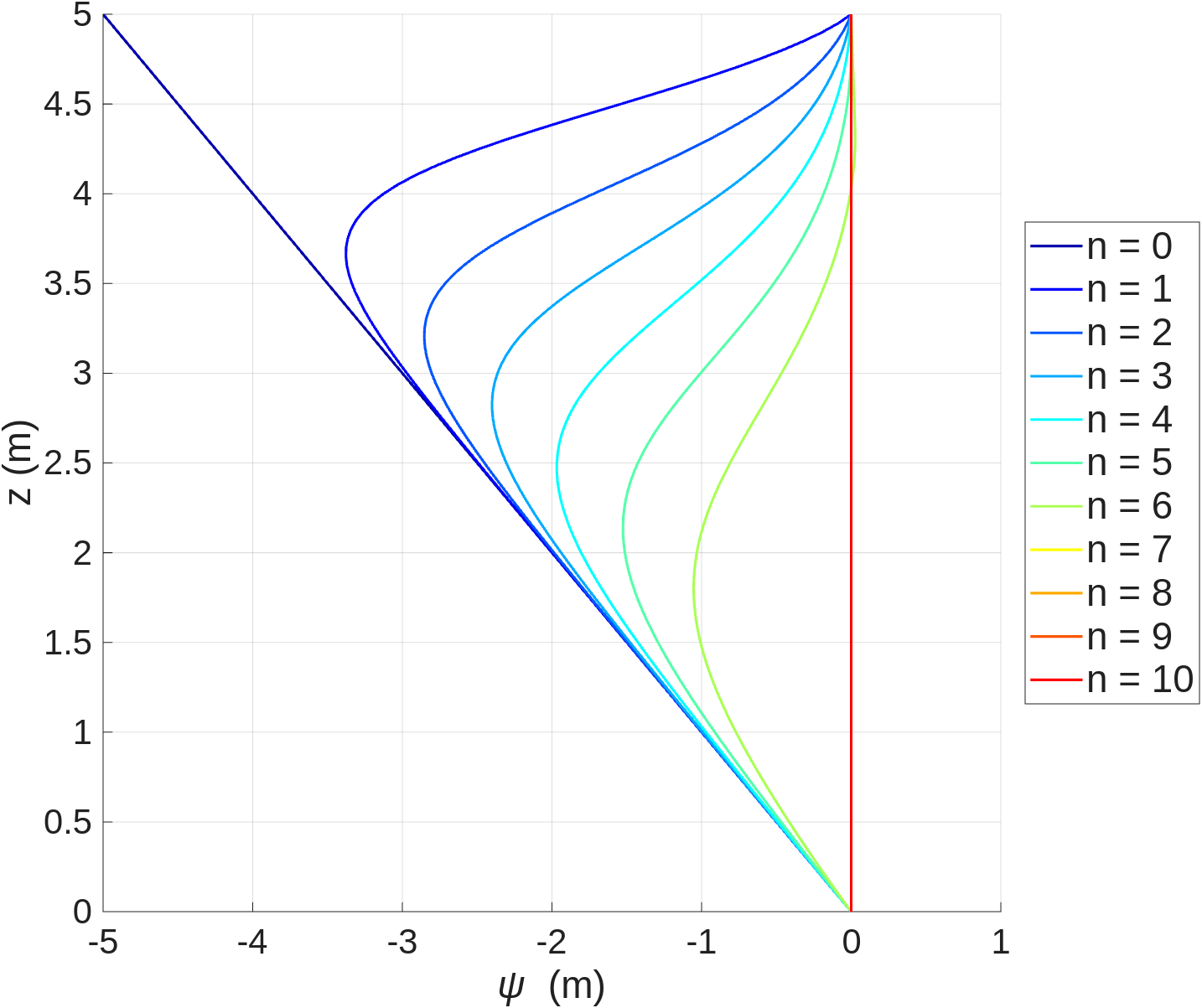}
    \caption{Test 3: Numerical solution extracted along the vertical line $x = 0.05$ for each $n = 0, \dots, N_{\finaltime} = 10$. \ref{test3:case3}.}
    \label{fig:test3:case3}
\end{figure}

The case \ref{test3:case1} is characterized by a rainfall-to-saturated-permeability ratio smaller than $1.0$, indicating that the rainfall intensity is insufficient to saturate the soil surface. Figure \ref{fig:test3:case1} shows the numerical solution extracted along the vertical line $x = 0.05$ for each time step $n = 0, \dots, N_{\finaltime}$. As observed in the figure, the solution never reaches saturation at the top boundary ($z = 5\ \rm{m}$). Consequently, Nitsche's method consistently enforces the Neumann boundary condition \eqref{test3:neuman_condition}, independently of the values of $s_{\mathrm{ph}}$ and $s_{\mathrm{fl}}$ defined in \eqref{eq:nitshce:ratio_terms}.

The numerical solution corresponding to case \ref{test3:case2} is reported in Figure \ref{fig:test3:case2}. Unlike the previous scenario, Nitsche's method switches from the Neumann boundary condition \eqref{test3:neuman_condition} to the homogeneous Dirichlet condition, $\ph = 0$, at time step $n = 7$, when the top boundary becomes saturated. At the final simulation time, the soil is fully saturated. In this test case, the parameters are set to $s_{\mathrm{ph}} = 0.001$ and $s_{\mathrm{fl}} = 0$. Immediately after the switching event, the Picard iteration terminates (after reaching the maximum allowed number of iterations $100$) with a residual of the order of $10^{-5}$, corresponding to the lowest accuracy attained during the simulation. By contrast, at both the beginning and the end of the simulation, the method achieves a higher accuracy in very few iterations, with residuals of approximately $10^{-7}$. 

The last case, \ref{test3:case3}, represents the most challenging scenario from a numerical standpoint, as the rainfall-to-saturated-permeability ratio exceeds $1.0$. In this regime, the rainfall intensity is greater than the soil infiltration capacity, causing the top boundary to saturate immediately. As a result, the switching from the Neumann to the homogeneous Dirichlet condition occurs during the first non-linear iteration, as illustrated by the numerical solution in Figure \ref{fig:test3:case3}. For every time step $n$, the Picard iteration terminates with a residual of the order of $10^{-7}$ in a few iterations and reaches machine precision toward the end of the simulation, when the solution converges to the trivial zero solution.

In all the test cases considered, the penalty parameter is fixed to $\penalt = 1$. We observe that the value of this parameter does not significantly affect the overall behaviour of the simulations in this numerical experiment.

\subsection{Test 4: The hydro-mechanical model}

In this experiment, we assess the performance of the method when solving the semi-coupled hydro-mechanical problem of soil stability. For this purpose, we consider a benchmark problem dealing with rainfall on a hill-slope, proposed in \cite{Lu2012, Moradi2018, Abbasov2024}.

In this numerical experiment, the body force vector $\bb = \bb(\ph)$ in the linear momentum equilibrium \eqref{eq:momentum_balance} depends on the water pressure-head and is expressed by
\begin{equation}
   \bb(\ph) = -\left(\rho_s (1 - \wc_s) + \rho_w \wc(\ph) \right) g \ee_z,
\end{equation}
where $\rho_s\ [M L^{-3}]$ represents the density of soil grains.  

\begin{figure}[!ht]
    \centering
    \includegraphics[width=0.7\linewidth]{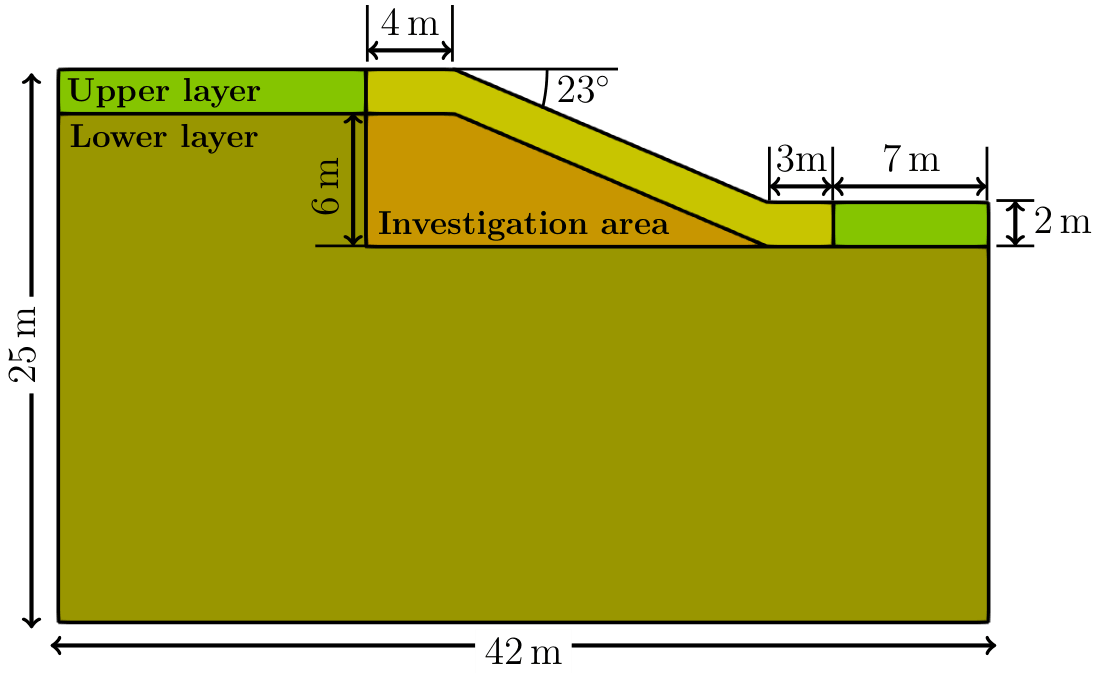}
    \caption{Test 4: Computational domain and investigation area.}
    \label{fig:test4:domain}
\end{figure}

We consider a two-layered domain that extends over $42\rm{m}$ in length and $25 \rm{m}$ in depth with a $23$-degree slope in the middle. Nonetheless, the risk of a landslide is analysed in a smaller investigation area located in the middle. Indeed, the simulations are performed for the whole domain to reduce the effect of the boundary conditions on stability assessments. The computational domain, along with all the relevant information, is reported in Figure \ref{fig:test4:domain}. The point $(0,0)$ is set at the bottom left corner of the domain.

The upper layer is $2\rm{m}$ thick along the entire length of the domain. The two layers exhibit distinct hydraulic properties, described by the Brooks-Corey constitutive laws \eqref{eq:brooks:parameters}. The corresponding parameters are listed in Table \ref{tab:test4:paramaters}, along with the mechanical properties of the soil. 
\begin{table}[!ht]
\centering
\caption{Test 4: Hydro-mechanical parameters. Parameters reported with a single value are assumed to be the same for both the upper and lower layers.}
\label{tab:test4:paramaters}
\begin{tabular}{@{}cllll@{}}
\toprule
\textbf{Symbol} & \textbf{Parameter Name}          & \multicolumn{1}{c}{\textbf{Units}} & \multicolumn{1}{c}{\textbf{Upper layer}} & \multicolumn{1}{c}{\textbf{Lower layer}} \\ \midrule
$\wc_s$         & Saturated water content           & $-$                                & 0.412 & 0.385                               \\
$\wc_r$         & Residual water content           & $-$                                & 0.041 & 0.09                             \\
$K_s$           & Saturated hydraulic conductivity & $\rm{m}\ \rm{s}^{-1}$                          & 7.2e-6 & 1.7e-7                           \\
$\alpha_{BC}$        & Brooks-Corey fitting parameter  & $\rm{m}^{-1}$                           & 6.8 & 2.7                               \\
$n_{BC}$             & Brooks-Corey fitting parameter  & $-$                                & 0.322 & 0.131                             \\ 
$l_{BC}$             & Brooks-Corey fitting parameter  & $-$                                & 1 &                          \\
$\rho_w$        & Density of water       & $\rm{kg}\ \rm{m}^{-3}$                         & 1000                     &            \\
$g$        & Acceleration of gravity      & $\rm{m}\ \rm{s}^{-2}$                         & 10                     &            \\
$\rho_s$        & Density of soil grains       & $\rm{kg}\ \rm{m}^{-3}$                         & 2636                    &             \\
$E$             & Young module                     & $\rm{kPa}$                              & 10000            &                     \\
$\nu$           & Poisson ratio                    & $-$                                & 0.35               &                \\
$\phi'$         & Friction angle                   & $\circ$                            & 35                    &             \\
$c'$            & Effective cohesion               & $\rm{kPa}$                              & 3 & 6                                 \\ \bottomrule
\end{tabular}
\end{table}

The simulation is performed for two periods of rainfall:
\begin{itemize}
    \item a first period $(-10\rm{years}, 0)$ of 10 years with a low-intensity rainfall of $600\ \rm{mm}\ \rm{year}^{-1}$, starting from the following hydrostatic pressure
    \begin{equation}
        \ph(-10\ {\rm{years}}, x, z) = \overline{\ph}(x,z) \coloneq 8 - z.
        \label{eq:test4:initialcondition}
    \end{equation}
    This period is simulated to predict the natural state of soil \cite{Abbasov2024}.
    \item a second period $(0, 15\rm{h})$, taking the final state of the previous period as initial condition, with a high-intensity rainfall event, i.e. $20\ \rm{mm}\ \rm{h}^{-1}$, during $15\rm{h}$.
\end{itemize}

\begin{figure}[!ht]
    \centering
    \includegraphics[width=0.8\linewidth]{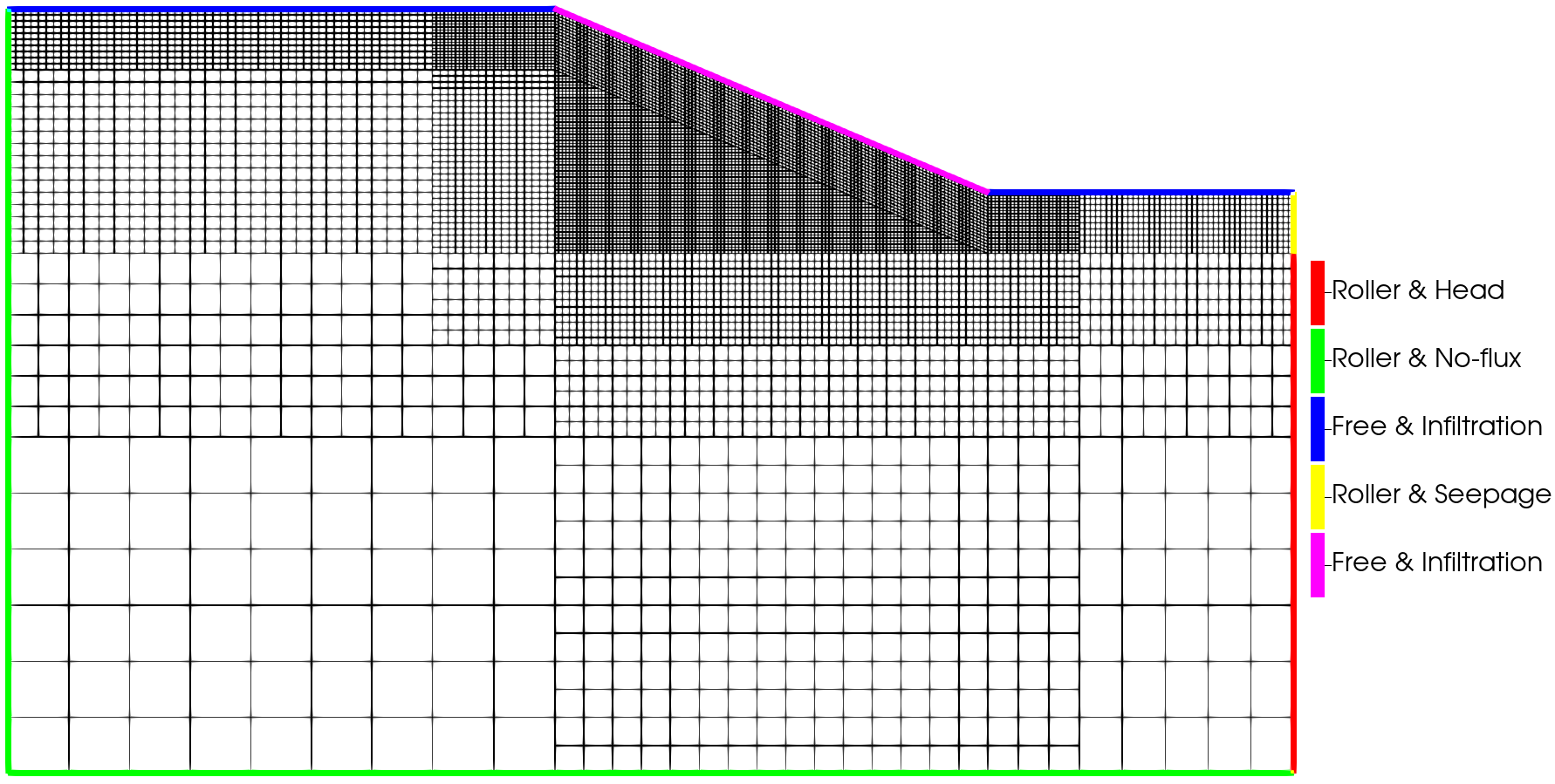}
    \caption{Test 4: Boundary conditions and computational mesh.}
    \label{fig:test4:bc}
\end{figure}

The imposed boundary conditions are summarized in Figure \ref{fig:test4:bc}. Concerning the Richards' equation, on the left and bottom surfaces, we impose no-flux boundary conditions, whereas we consider a rainfall flux on the top surface. We observe that the normal rain flux on the slope (highlighted in magenta in Figure \ref{fig:test4:bc}) is equal to the one on the horizontal surface, multiplied by the cosine of the inclination angle \cite{Abbasov2024}.  We set a fixed pressure-head equal to $\overline{\ph}$, defined in \eqref{eq:test4:initialcondition}, on the bottom part of the right border (highlighted in red in Figure \ref{fig:test4:bc}), whereas a seepage condition is imposed on the remaining part of this boundary. Finally, the top surface is considered free of stress, whereas roller boundary conditions are imposed at the bottom, right, and left surfaces.

The simulation domain was discretized using a graded mesh, with cell sizes increasing with depth (Figure \ref{fig:test4:bc}). A finer discretization was adopted near the ground surface to accurately capture the highly dynamic hydrological processes associated with rainfall infiltration, whereas a coarser mesh was employed in deeper layers, where hydrological conditions evolve more gradually. Thanks to the capability of the VEM framework to naturally handle hanging nodes, a conforming graded mesh can be easily constructed while preserving the layer subdivision and accurately fitting the geometry of the investigation domain \cite{Vicini2024}. A constant uniform time discretization was considered, with $N_{\finaltime}=1000$ time steps for the first simulation period and $N_{\finaltime}=2000$ time steps for the second period.

Figures \ref{fig:test4:water_content} and \ref{fig:test4:lfs} show the spatial distribution of water content and Local Factor of Safety values within the investigation area at three representative time instants. The first corresponds to $t=0\mathrm{h}$, marking the end of the first simulation period and the onset of the second. The second is $t=7.5\mathrm{h}$, representing the midpoint of the rainfall event simulated during the second period. The final snapshot is taken at $t=15\mathrm{h}$, corresponding to the end of the simulation.

During the first simulation period, rainfall progressively infiltrated the soil profile, penetrating the upper layer and redistributing into the lower layer. At the same time, water accumulated in the downslope region, near the toe of the upper layer. This configuration corresponds to a hydrological equilibrium state reached after approximately two years of simulation. The LFS distribution at $t=0\mathrm{h}$ indicates that this equilibrium condition does not exhibit any potential failure zones within the slope.

The effects of the intense rainfall event simulated during the second period are clearly visible in both the water content and LFS fields. After $7.5\mathrm{h}$ of rainfall, the top part of the soil begins to approach saturation. As the degree of saturation increases, a potential failure zone starts to develop near the ground surface, as highlighted by the reduction in LFS values. At the end of the simulation ($t=15\mathrm{h}$), the soil profile is almost completely saturated, and the potentially unstable region has expanded toward the lower part of the slope and near its toe. The predicted evolution of both the hydraulic and stability responses closely matches the results obtained using {COMSOL{\textregistered}} and reported in the reference study \cite{Abbasov2024}, confirming the viability of the proposed approach to accurately reproduce the semicoupled hydro-mechanical model for slope stability.

\begin{figure}[!ht]
    \centering
    \begin{subfigure}{0.9\textwidth}
        \includegraphics[width=1\linewidth]{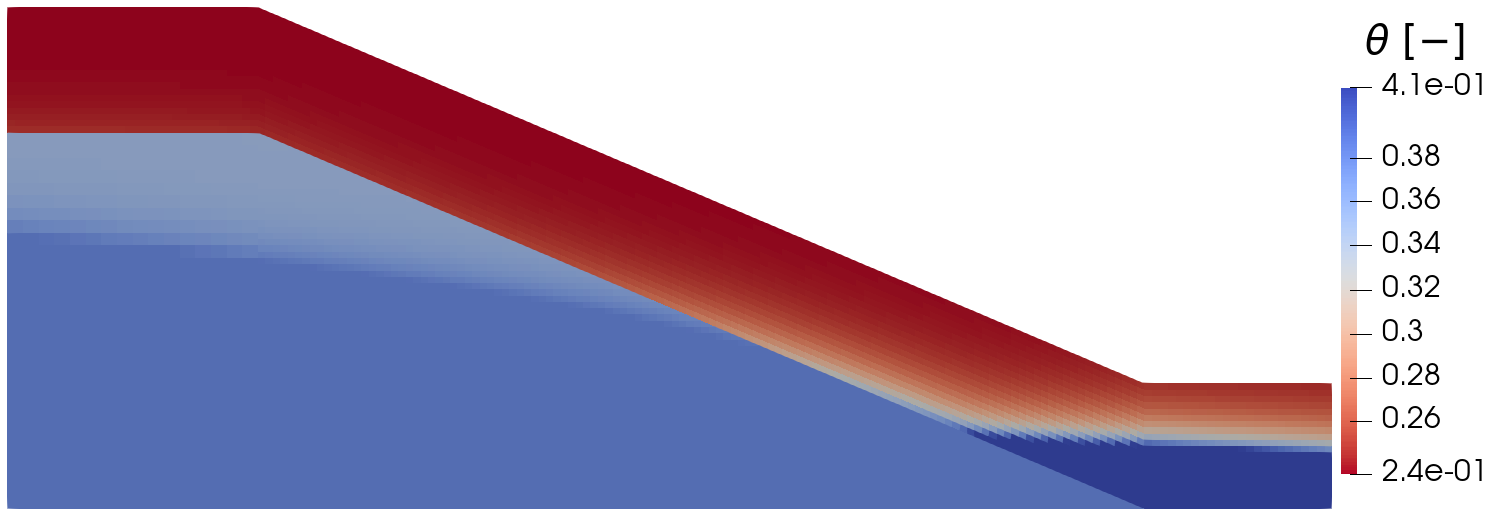}
        \caption{$t = 0\rm{h}$}
    \end{subfigure}
    \begin{subfigure}{0.9\textwidth}
        \includegraphics[width=1\linewidth]{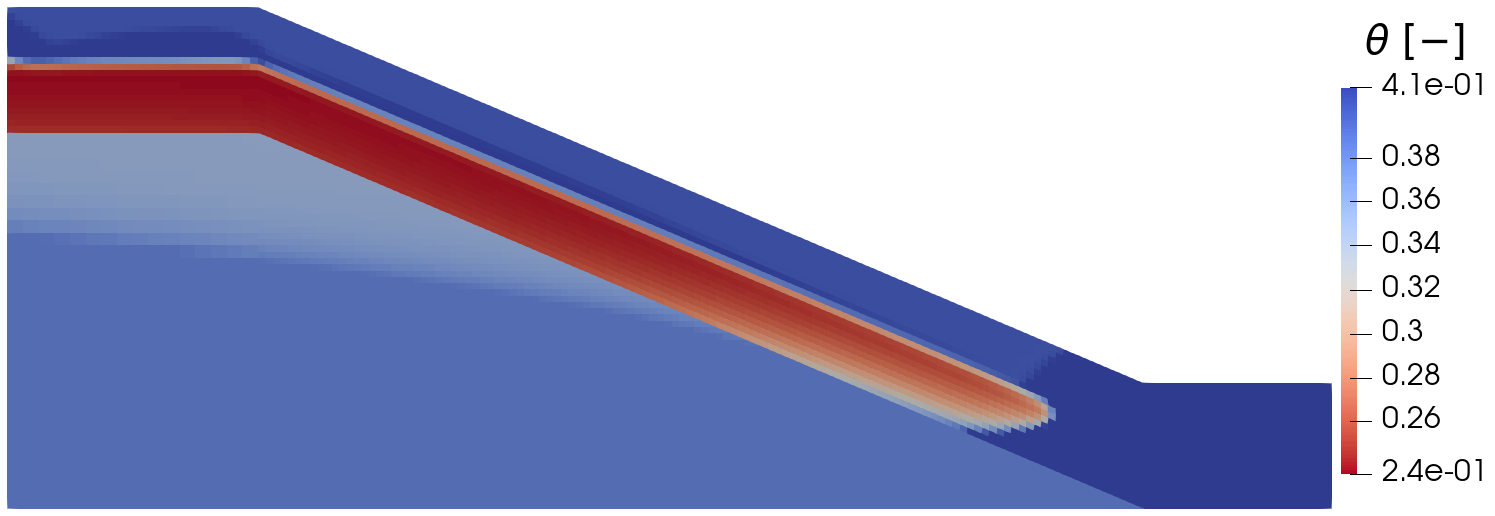}
        \caption{$t = 7.5\rm{h}$}
    \end{subfigure}
    \begin{subfigure}{0.9\textwidth}
        \includegraphics[width=1\linewidth]{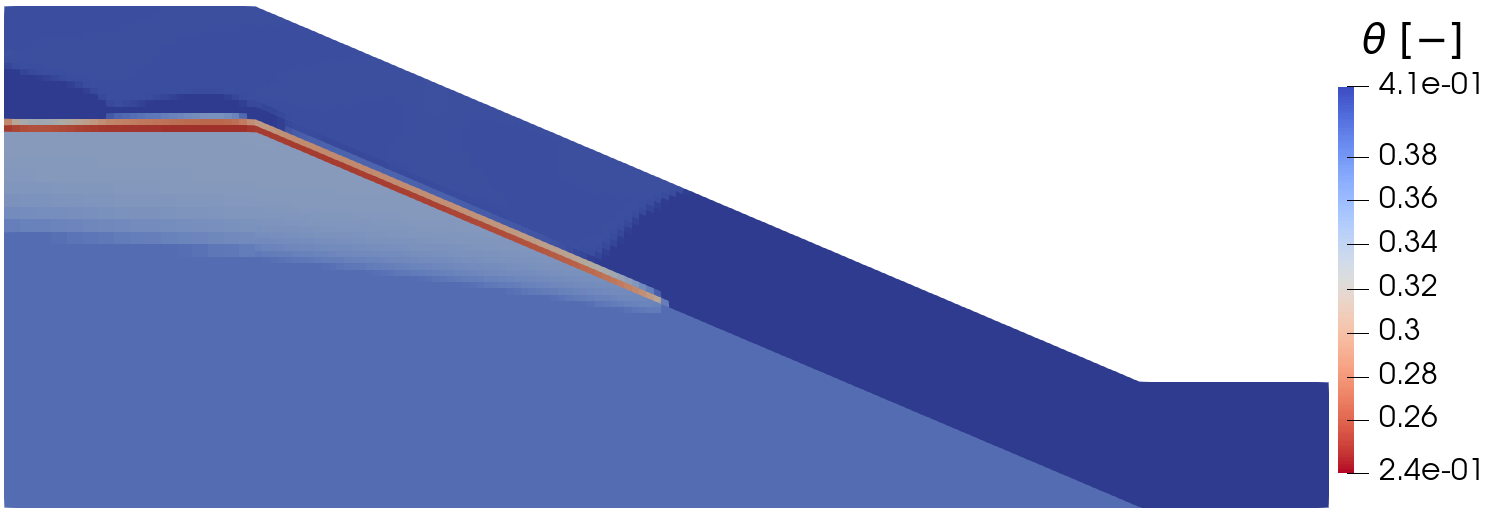}
        \caption{$t = 15\rm{h}$}
    \end{subfigure}
    \caption{Test 4: Discrete water content $\wc$ values within the investigation area at different time steps.}
    \label{fig:test4:water_content}
\end{figure}

\begin{figure}[!ht]
    \centering
    \begin{subfigure}{0.9\textwidth}
        \includegraphics[width=1\linewidth]{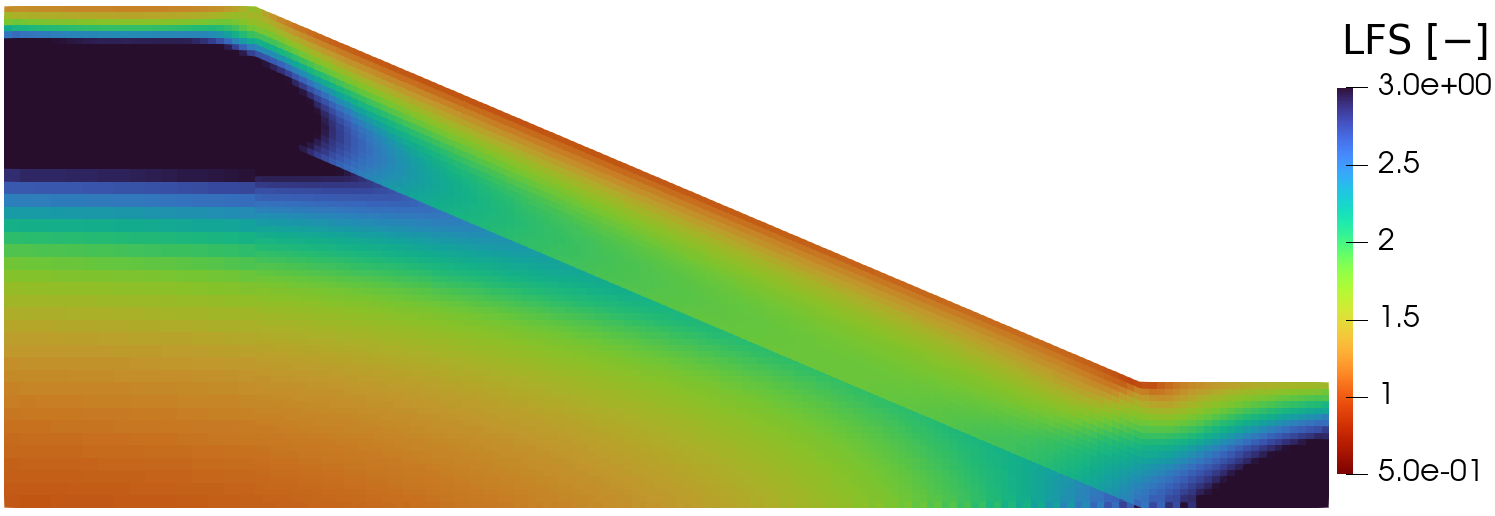}
        \caption{$t = 0\rm{h}$}
    \end{subfigure}
    \begin{subfigure}{0.9\textwidth}
        \includegraphics[width=1\linewidth]{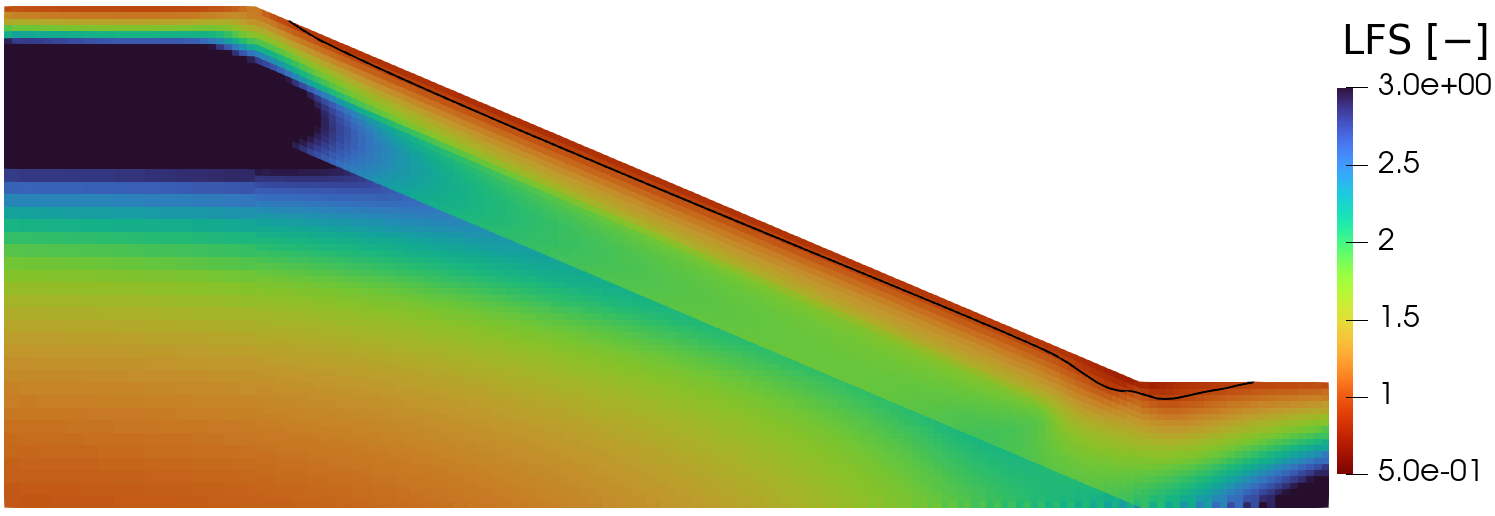}
        \caption{$t = 7.5\rm{h}$}
    \end{subfigure}
    \begin{subfigure}{0.9\textwidth}
        \includegraphics[width=1\linewidth]{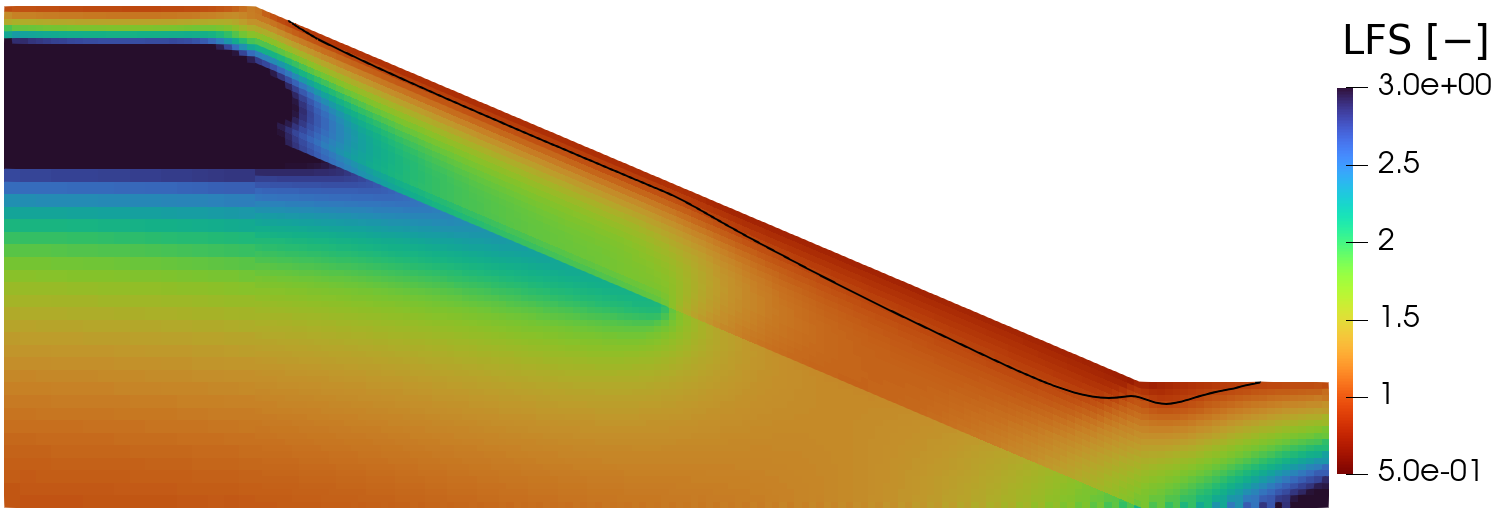}
        \caption{$t = 15\rm{h}$}
    \end{subfigure}
    \caption{Test 4: Discrete LFS values within the investigation area at different time steps. The black curve denotes the surface $\rm{LFS}=1$.}
    \label{fig:test4:lfs}
\end{figure}

\begin{remark}
    In this test case, to avoid overshoot/undershoot of the pressure-head in cells characterized by partial saturation, we discretize the permeability $K$ and the capacity term $C$ as element-wise constant in \eqref{eq:defah}-\eqref{eq:defmh}, i.e. $K_h = K(\Proj{0}{E} \ph_h)$ and $C_h = C(\Proj{0}{E} \ph_h)$.
\end{remark}

\section{Conclusion}

In this paper, we introduce the stabilization-free Virtual Element Method for the spatial discretization of the non-linear hydro-mechanical model that couples the Richards' equation with a linear elastic problem to assess soil stability. 
Seepage-face and infiltration boundary conditions are introduced into the model through Nitsche's method, allowing for the automatic transition between Neumann and Dirichlet boundary conditions according to the local hydraulic state. A theoretical analysis is established to show the stability of the resulting spatial discretization.

Moreover, the method is combined with a mass-lumping strategy, which eliminates the need for stabilization terms also in the storage contribution while mitigating spurious oscillations at the infiltration front. Time discretization is performed using the backward Euler scheme, whereas the non-linearities arising from Richards' equation are handled through a Picard iterative procedure.

Several benchmark experiments are simulated to show the performance of the proposed methodology and to demonstrate its viability and robustness in the simulation of semi-coupled hydro-mechanical problems.

\section*{Acknowledgements}

The author S.B. kindly acknowledges partial financial support provided by European Union through project Next Generation EU, M4C2, PRIN 2022 PNRR project P2022BH5CB\_001 ``Polyhedral Galerkin methods for engineering applications to improve disaster risk forecast and management: stabilization-free operator-preserving methods and optimal stabilization methods'', and by PNRR M4C2 project of CN00000013 National Centre for HPC, Big Data and Quantum Computing (HPC) (CUP: E13C22000990001). The authors F.M. and G.T. kindly acknowledge the financial support provided by INdAM-GNCS Project ``Metodi numerici politopali stabilization-free e neural-based per problemi accoppiati e non lineari'' (CUP: E53C25002010001). 

\bibliographystyle{IEEEtranDOI}
\bibliography{biblio.bib}

\end{document}